\documentclass[12 pt]{article}
\usepackage{graphicx,mathtools,amsmath,amssymb,amsthm,enumerate, hyperref, xcolor,mathrsfs}
\usepackage{marvosym, fontawesome}
\usepackage[indent = 20 pt]{parskip}

\usepackage{url}
\usepackage{enumitem}

\newtheorem{theorem}{Theorem}[section]

\newtheorem{lemma}[theorem]{Lemma}
\newtheorem{proposition}[theorem]{Proposition}

\newtheorem{letterthm}{Theorem}
\newtheorem{lettercor}[letterthm]{Corollary}
\newtheorem{letterques}[letterthm]{Question}

\theoremstyle{definition}
\newtheorem{definition}[theorem]{Definition}

\newtheorem{example}[theorem]{Example}

\theoremstyle{remark}
\newtheorem{remark}[theorem]{Remark}

\usepackage{etoolbox}
\preto\theorem{\medskip}

\DeclareMathOperator{\rad}{rad}
\newcommand{\md}{\rm mod}

\DeclareMathOperator{\ad}{ad}

\newcommand{\Flow}{\text{Flow}_{\text{cts}}}
\newcommand{\cc}{\text{cc}}

\newcommand{\mr}{\mathfrak{r}}

\newcommand{\mg}{\mathfrak{g}}
\newcommand{\mv}{\mathfrak{v}}

\newcommand{\mh}{\mathfrak{h}}
\newcommand{\ms}{\mathfrak{s}}
\newcommand{\ma}{\mathfrak{a}}
\newcommand{\mk}{\mathfrak{k}}

\newcommand{\cK}{\mathcal{K}}
\newcommand{\cB}{\mathcal{B}}

\newcommand{\cF}{\mathcal{F}}

\newcommand{\cH}{\mathcal{H}}

\newcommand{\cR}{\mathcal{R}}

\newcommand{\cS}{\mathcal{S}}
\newcommand{\cU}{\mathcal{U}}

\newcommand{\C}{\ensuremath{\mathbb{C}}}
\newcommand{\N}{\ensuremath{\mathbb{N}}}
\newcommand{\F}{\ensuremath{\mathbb{F}}}
\newcommand{\R}{\ensuremath{\mathbb{R}}}
\newcommand{\T}{\ensuremath{\mathbb{T}}}
\newcommand{\Z}{\ensuremath{\mathbb{Z}}}

\newcommand{\Img}{\text{Im}}

\DeclareMathOperator{\MP}{mp}

\newcommand{\actson}{\curvearrowright}

\DeclareMathOperator{\id}{id}
\DeclareMathOperator{\Ad}{Ad}
\DeclareMathOperator{\Aut}{Aut}
\DeclareMathOperator{\Inn}{Inn}

\DeclareMathOperator{\Tr}{Tr}
\DeclareMathOperator{\at}{at}
\DeclareMathOperator{\ct}{ct}

\DeclareMathOperator{\Out}{Out}

\DeclareMathOperator{\SL}{SL}

\DeclareMathOperator{\GL}{GL}
\DeclareMathOperator{\reg}{reg}

\DeclareMathOperator{\Stab}{Stab}

\DeclareMathOperator{\Obs}{Obs}

\DeclareMathOperator{\supp}{supp}

\DeclareMathOperator{\cts}{cts}
\DeclareMathOperator{\ns}{ns}

\usepackage[doi=false, url =false , isbn = false,style=alphabetic,sorting=nyt, backend = biber, maxcitenames=50, maxalphanames = 5, maxnames=50]{biblatex}
\makeatletter
\def\thanks#1{\protected@xdef\@thanks{\@thanks
        \protect\footnotetext{#1}}}
\makeatother

\begin{document}

\title{Non-uniqueness of continuous trace-scaling flows on certain full factors}
\author{
  Soham Chakraborty\textsuperscript{1} \and
  Sergiu Novac\textsuperscript{2}
}

\date{\today}

\setlength{\parindent}{0em}

\maketitle 

\begingroup
\renewcommand\thefootnote{\arabic{footnote}}
\footnotetext[1]{\hspace{0 em} \faMapMarker : Départment de Mathématiques et Applications, École Normale Supérieure, 45 Rue d'Ulm, 75005 Paris, France. \Letter: \texttt{soham.chakraborty@ens.psl.eu}}

\footnotetext[2]{\hspace{0 em} \faMapMarker : Department Wiskunde, KU Leuven, Celestijnenlaan 200B, 3001 Leuven, Belgium.  \Letter: \texttt{sergiu.novac@kuleuven.be}}
\endgroup

\begin{abstract}\noindent
    We show that if $M$ is a full II$_\infty$ factor such that the fundamental group of a finite corner is $\R^*_+$ and $\Out(M)$ is a locally compact second countable group, then $M$ admits either a unique continuous trace scaling flow or a continuum of pairwise distinct ones, up to cocycle conjugacy. Equivalently, there is either a unique III$_1$ factor or a continuum of pairwise non-isomorphic III$_1$ factors with continuous core $M$. We apply this dichotomy result to factors constructed in \cite{Deprez12} and \cite{Chakraborty25} where $\Out(M)$ is a Lie group. We also give two independent constructions of full factors admitting a continuum of pairwise distinct trace-scaling flows where the outer automorphism group is unknown.
\end{abstract}

\section{Introduction}

After the development of Tomita-Takesaki theory, the works of Connes and Takesaki (\cite{Connes73}, \cite{Takesaki73}, \cite{Cones-Takesaki}) associated to every type III$_1$ factor $M$ its canonical \textit{non-commutative flow of weights}. This consists of the continuous core $c(M)$ which is a II$_\infty$ factor, together with a strongly continuous trace-scaling flow $\alpha: \R \actson c(M)$ such that the crossed product is again isomorphic to $M$. They show that the non-commutative flow of weights is in fact an isomorphism of categories between III$_1$ factors and II$_\infty$ factors with continuous trace-scaling flows. In particular, a given II$_\infty$ factor $B$ arises as a continuous core if and only if it admits a continuous trace-scaling flow. Moreover, two type III$_1$ factors $M$ and $N$ are isomorphic if and only if $c(M) = c(N)$ and the corresponding trace-scaling flows are cocycle conjugate. 

Existence and uniqueness (up to cocycle conjugacy) of a continuous trace-scaling flow on a II$_\infty$ factor are difficult problems in general. The first obstruction is the fundamental group of a corner. If $p$ is a finite trace projection in $B$, then existence of such a flow forces the fundamental group $\cF(pBp)$ to be $\R^*_+$. Conversely, $\cF(pBp) = \R^*_+$ implies that for all $t\in \R$, there is an automorphism of $B$ that scales the trace by $e^{-t}$. However this pointwise realization does not automatically provide a homomorphism $\R \rightarrow \Aut(B)$, and a fortiori does not provide a continuous trace-scaling flow. In fact it was shown in \cite{PopaVaes10} that there is a II$_1$ factor $N$ with $\cF(N) = \R^*_+$ but $N \otimes \cB(\cH)$ does not admit a continuous trace-scaling flow.

The first interesting case that depicts uniqueness is of the injective II$_\infty$ factor $R_{0,1}$ (which is unique by \cite{Connes76}). In the seminal work of Haagerup (\cite{Haagerup87}), it was shown that there is a unique injective type III$_1$ factor. Equivalently, there is a unique continuous trace-scaling flow on $R_{0,1}$. Of course one can use this to construct continuous trace-scaling flows on any II$_\infty$ McDuff factor. In fact it follows from \cite{Houdayer07} that certain McDuff II$_\infty$ factors even admit a continuum of pairwise non-cocycle conjugate continuous trace-scaling flows. This depicts a stark contrast from the uniqueness result of Haagerup.   

On the other hand for full factors, even existence is subtle. The first major development happened in \cite{Radulescu} and \cite{Radulescu92b}, where it was shown that $\cF(L(\F_\infty)) = \R^*_+$ and in fact $L(\F_\infty) \otimes \cB(\cH)$ admits a continuous trace-scaling flow. It is still open if this trace-scaling flow of R\u adulescu is the unique one in $L(\F_\infty) \otimes \cB(\cH)$. Dykema and R\u adulescu subsequently used free-product constructions with an unbounded semicircular element to produce new examples of II$_\infty$ factors admitting continuous trace-scaling flows in \cite{DykemaRadulescu}. In particular this applies to stabilizations of full II$_1$ factors without the Haagerup approximation property. We also mention here that the case of $L(\F_2) \otimes \cB(\cH)$ remains tied to the free group factor problem. A trace-scaling flow here would imply that $\cF(L(\F_2)) = \R^*_+$, which would imply that the free group factors are all isomorphic by the results of Dykema \cite{Dykema94} and R\u adulescu \cite{Radulescu94}.

For convenience of notation, let us denote the set of continuous trace-scaling flows on $M$ by $\Flow(M)$ and the equivalence relation given by cocycle conjugacy, by $\sim_{\cc}$. In this article, we investigate if there exists a full II$_\infty$ factor $M$ with prescribed cardinality of the set $\Flow(M)/\sim_{\cc}$. In particular, we construct full factors satisfying $|\Flow(M)/\sim_{\cc}| = 2^{\aleph_0}$. As usual we denote the outer automorphism group of a full factor $M$ by $\Out(M)$ and the fundamental group  of a II$_1$ factor $N$ by $\cF(N)$. The main result of this article is: 
\medskip
\begin{letterthm}
\label{Thm: main thm dichotomy}
 Suppose that $M$ is a full II$_\infty$ factor with $\cF(pMp) = \R^*_+$ for a finite trace projection $p \in M$. If $\Out(M)$ is a locally compact second countable group, then either $|\Flow(M)/ \sim_{\cc}| = 1$ or $|\Flow(M)/ \sim_{\cc}| = 2^{\aleph_0}$.
\end{letterthm}

We prove Theorem \ref{Thm: main thm dichotomy} in three main steps. The first step is Proposition \ref{Prop: flows and splits} where we transform the problem into an abstract group theoretic problem. Let $\md: \Out(M) \rightarrow \R^*_+$ be the usual module homomorphism and $\delta = - \log(\md): \Out(M) \rightarrow \R$, we notice that every trace-scaling flow gives a continuous splitting of the map $\delta$. Abstractly, for a continuous surjective group homomorphism $\delta: G \rightarrow \R$, we denote the Polish space of all continuous splittings of $\delta$ by $\mathfrak{S}(\delta)$. Then $G$ acts continuously on $\mathfrak{S}(\delta)$ by conjugation. We show that when $G = \Out(M)$ for a full factor $M$ and $\delta$ is as above then there is a bijection between $\mathfrak{S}(\delta)/G$ and $\Flow(M)/\sim_{\cc}$. The main ingredient is Sutherland's results on Borel lifts of $\R$-kernels with trivial obstructions, as in \cite{Sutherland_cohomology}. Theorem \ref{Thm: main thm dichotomy} then follows from the following result, which is the main technical effort of this article: 
\medskip
\begin{letterthm}
\label{Main theorem: general lc group case}
    Let $G$ be a locally compact second countable group and $\delta: G \rightarrow \R$ be a continuous surjective homomorphism. Let $\mathfrak{S}(\delta)$ be the Polish space of all continuous splittings of $\delta$ and consider the continuous action $G \actson \mathfrak{S}(\delta)$ by conjugation (see Lemma \ref{Lemma: S(delta) Polish space and G action continuous} for continuity). Then $|\mathfrak{S}(\delta)/G| = 1$ or $2^{\aleph_0}$. 
\end{letterthm}

We first prove Theorem \ref{Main theorem: general lc group case} when $G$ is a Lie group. In Proposition \ref{Prop: number of orbits of a} we transform this into a Lie algebraic problem. Indeed let $\delta: G \rightarrow \R$ as before and let $\mg$ be the Lie algebra of $G$. Let $\mk$ be the kernel of the differential $d\delta$ and $\ma$ be $d\delta^{-1}(\{1\})$, which is an affine subspace of $\mg$. The affine hyperplane $\ma$ is invariant under the $\Ad$-action of $G$ and we show that there is a bijection between $\mathfrak{S}(\delta)$ and $\ma$ that is $G$-equivariant with respect to the conjugation action and $\Ad$ respectively. The next step is to use several tedious but standard Lie algebraic techniques to show that if $|\ma/G| \geq 1$, then either $|\ma/G| = 1$ or $|\ma/G| = 2^{\aleph_0}$. Finally we deduce the theorem for all locally compact groups from the Lie group case by an application of the Gleason-Yamabe theorem. We apply Theorem \ref{Thm: main thm dichotomy} to the full II$_\infty$ factors constructed in \cite{Deprez12} and \cite{Chakraborty25} where the outer automorphism groups are in fact Lie groups. In Propositions \ref{Prop: Chakraborty construction} and \ref{Prop: Deprez construction}, we show: 
\medskip
\begin{lettercor}
\label{Corr: main corollary}
    Let $M_C = L(\cR_C)$ be the full II$_\infty$ factor constructed in \cite[Definition 4.1]{Chakraborty25} and for $n \geq 1$, let $M_n$ be the full II$_\infty$ factors which are amplifications of the II$_1$ factors constructed in \cite[Theorem 8.4]{Deprez12}. Then we get: 
    \begin{enumerate}
        \item $\Out(M_C) = \R$ and $|\Flow(M_C)/\sim_{\cc}| = 1$,
        \item $\Out(M_n) = \GL(n,\R)$ and $|\Flow(M_n)/\sim_{\cc}| = 1$ for $n =1$ and $2^{\aleph_0}$ for $n \geq 2$. 
    \end{enumerate}
\end{lettercor}
 
In the next part of the article, we give two different constructions of full II$_\infty$ factors with a continuum of pairwise distinct continuous trace-scaling flows, in cases where the outer automorphism group is unknown and difficult to compute. Recall that a non-amenable countable ICC group is in class $\mathcal{C}$ of Ozawa and Popa (\cite{Ozawa_Popa}) if it is either a subgroup of a hyperbolic group or a discrete subgroup of a simple rank 1 Lie group. In particular, all free groups are in class $\mathcal{C}$ and hence the following result applies to $G = \F_\infty$.  
\medskip
\begin{letterthm}
\label{Thm: main thm tensor case}
    Let $G$ be a group in class $\mathcal{C}$ such that there is a continuous trace-scaling flow on $L(G) \otimes \cB(\cH)$. Let $M = L(G \times G) \otimes \cB(\cH)$, then $|\Flow(M)/ \sim_{\cc}| = 2^{\aleph_0}$. 
\end{letterthm}

The main ingredient in the proof of Theorem \ref{Thm: main thm tensor case} is a semifinite version of the unique prime factorization result of \cite{Ozawa_Popa} for group von Neumann algebras of products of groups from class $\mathcal{C}$ (see Proposition \ref{Prop: unique prime factors for semifinite factors}). Recall that a II$_\infty$ factor $M$ is called \textit{prime} if it cannot be decomposed as a tensor product of two II$_\infty$ factors. In the next result we construct a full prime II$_\infty$ factor with a continuum of pairwise distinct trace-scaling flows. 
\medskip
\begin{letterthm}
\label{Thm: main thm prime case}
    Let $P_0 = (L(\F_\infty) \otimes L^\infty[0,1]) \ast L(\F_\infty)$ and let $P = P_0 \otimes \cB(\cH)$. Then $P$ is a full II$_\infty$ prime factor with $|\Flow(P)/\sim_{\cc}| = 2^{\aleph_0}$. 
\end{letterthm}

To prove Theorem \ref{Thm: main thm prime case}, we explicitly construct an uncountable family of free Araki-Woods factors. Indeed we observe that letting $R_{\reg}$ be the free Araki-Woods factor corresponding to the left regular representation, and $T_\lambda$ be the free Araki-Woods factors of type III$_\lambda$, the free products $T_\lambda \ast R_{\reg}$ with respect to their usual quasi-free states are pairwise non-isomorphic as $\lambda$ varies. One way to see this is to consider their spectral measures and use the classification of Houdayer-Shlyakhtenko-Vaes as in \cite{HouShlyaVaes} (see Lemma \ref{Lemma: non isomorphic factors}). The rest of the technical effort is to use the results of \cite{Shlyakhtenko04} to show that their continuous cores are isomorphic to $P$. We conclude by posing the following natural question that stems out of this article: 
\medskip
\begin{letterques}
 For which Polish groups $G$ and surjective continuous homomorphisms $\delta: G \rightarrow \R$, can one find a full II$_\infty$ factor $M$ such that the pairs $(G,\delta)$ and $(\Out(M), -\log (\md))$ are isomorphic?
\end{letterques}

\textbf{Convention:} In this article every von Neumann algebra is assumed to have a separable predual and every Lie group is assumed to be a locally compact second countable finite dimensional real Lie group, unless otherwise stated. 

\textbf{Acknowledgements:} S.C. is supported by the ERC advanced grant no. 101141693 titled \textit{Noncommutative ergodic theory of higher rank lattices}. S.C. would like to thank Cyril Houdayer for several nice discussions on this topic. 
S.N. is supported by FWO PhD Fellowship 11A1A26N of the Research Foundation Flanders. The authors would like to thank Cyril Houdayer, Yusuke Isono, Sergio Gir\'{o}n Pacheco and Stefaan Vaes for having a look at the draft and for some helpful comments. The authors are grateful to Stefaan Vaes for suggesting a deduction of Theorem \ref{Main theorem: general lc group case} for locally compact groups from the Lie group case.

\section{Preliminaries}
 \subsection{Full factors and trace-scaling flows}
 \label{Subsec: full factors}
    Let $M$ be a factor with a separable predual.
    The automorphism group $\Aut(M)$ is a Polish group with the so called \textit{u-topology} (see \cite[Chapter IX, Page 152]{TakesakiII}). The factor $M$ is called \textit{full} if the subgroup of inner automorphisms $\Inn(M)$ is a closed subgroup of $\Aut(M)$. The quotient is called the outer automorphism group of $M$ and is a Polish group exactly when $M$ is full. By \cite[Corollary 3.6]{Connes74}, $M$ is full if and only if for any free ultrafilter $\omega$ on $\N$, the central sequence algebra $M_\omega$ is trivial. It is easy to check that any type I factor is full. If $M$ is a II$_1$ factor, then $M$ is full if and only if $M$ does not have property $\Gamma$ of Murray and von Neumann. In particular, the hyperfinite II$_1$ factor is not full. By \cite[Proposition 3.9]{Connes74}, a factor of type III$_0$ is never full and by \cite[Corollary 3.9]{Connes74}, there are full factors of all other types. The unitary group $\cU(M)$ of a von Neumann algebra will be used throughout the article, as a Polish topological group with the strong -* topology. 

    Now suppose that $M$ is a II$_\infty$ factor with semifinite trace $\tau$. For an automorphism $\alpha \in \Aut(M)$, the Connes-Takesaki module is a continuous group homomorphism $\md: \Aut(M) \rightarrow \R^*_+$ where $\md(\theta)$ is the unique positive number such that $\tau \circ \theta = \md(\theta) \cdot \tau$. Since $\tau(uxu^*) = \tau(x)$ for a unitary $u$, $\md$ is trivial on the subgroup $\Inn(M)$. Hence if $M$ is full, we get a continuous group homomorphism $\md: \Out(M) \rightarrow \R^*_+$. 

    For a II$_1$ factor $B$, the fundamental group $\cF(B)$ is the subgroup of $\R^*_+$ given by $\{t > 0 \; | \; B^t \cong B\}$ where $B^t$ denotes the $t$-amplification of $B$. If $M$ is a full II$_\infty$ factor and $p$ is a projection with finite trace, then the image of $\md$ lands precisely inside $\cF(pMp)$. This gives a natural exact sequence: 
    \begin{align*}
        1 \rightarrow \Out(pMp) \rightarrow \Out(M) \rightarrow \cF(pMp) \rightarrow 1
    \end{align*}
    In this article \textit{a flow} on a II$_\infty$ factor $M$ will always mean an action of $\R$ on $M$ which is continuous for the u-topology on $\Aut(M)$ when viewed as a group homomorphism $\R\to \Aut(M)$. A flow $\alpha$ is called trace-scaling if $\tau \circ \alpha_t = e^{-t} \cdot \tau$ for all $t \in \R$. We denote by $\Flow(M)$ the set of all such continuous trace-scaling flows on $M$. Two flows $\alpha, \beta \in \Flow(M)$ are called cocycle conjugate if there is a continuous $1$-cocycle $u: \R \rightarrow \cU(M)$ satisfying $u_{s+t} = u_s \alpha_s(u_t)$ for all $s,t \in \R$ and an automorphism $\theta \in \Aut(M)$ satisfying: 
    \begin{align*}
        \theta^{-1} \circ \beta_t \circ \theta = \Ad(u_t) \circ \alpha_t \text{ for all } t \in \R
    \end{align*}
    If $\alpha$ and $\beta$ in $\Flow(M)$ are cocycle conjugate we shall denote this by $\alpha \sim_{\cc} \beta$. One can check that $\sim_{\cc}$ is indeed an equivalence relation on $\Flow(M)$. Notice that if $M$ admits a continuous trace-scaling flow, i.e., if $\Flow(M)$ is non-empty, then the image of $\md: \Out(M) \rightarrow \R^*_+$ is $\R^*_+$ and hence $\cF(pMp) = \R^*_+$.

    Finally note the map $\Flow(M)\to \text{Fun}(\mathbb{Q},\Aut(M))$ given by sending a flow $\alpha$ on $M$ to the function $\mathbb{Q}\ni q \mapsto \alpha_q$ is injective. Since Polish groups have cardinality at most that of the continuum, this shows also $|\Flow(M)| \leq2^{\aleph_0}$ - under our standing assumption that $M$ has separable predual. As such, whenever we claim a II$_\infty$ factor $M$ satisfies $|{\Flow(M)}/{\sim_{\cc}}|=2^{\aleph_0}$ we will only prove $|\Flow(M)/\sim_{\cc}|\geq 2^{\aleph_0}.$

    \subsection{Free Araki-Woods factors} 

    We recall some basic notions from the theory of free Araki-Woods factors. Let $K$ be a complex Hilbert space and recall that the Fock space $\cF(K)$ over $K$ is given by: 
    \begin{align*}
        \cF(K) = \C\Omega \oplus \bigoplus_{n \geq 1} K^{\otimes n}
    \end{align*}
    The vector $\Omega$ is called the \textit{vacuum} vector. For a vector $\xi \in K$, one has a \textit{left creation} operator  $\ell(\xi)$ given by: 
    \begin{align*}
        \ell(\xi)\Omega = \xi \text{ and } \ell(\xi) (\eta_1,...,\eta_n) = (\xi,\eta_1,...,\eta_n)
    \end{align*}
    One checks that $\ell(\xi)^* \ell(\eta) = \langle \eta,\xi \rangle$. If $K_{\R} \subset K$ is a real subspace, then for any $\xi \in K_{\R}$, we denote the operator $s(\xi) = \ell(\xi) + \ell(\xi)^*$. 

    Now suppose that $\cH_\R$ is a real separable Hilbert space and consider an orthogonal representation $U: \R \actson \cH_\R$. Let $\cH = \cH_\R \otimes _\R \C$ be the complexification of $\cH_\R$ and by abusing notation let us still denote the unique extension to a unitary  representation on $\cH$ by $U$. Then by Stone's theorem there is a self adjoint operator $A$ on $\cH$ such that $U_t = A^{it}$. Now consider the operator $j$ on $\cH$ given by: 
    \begin{align*}
        j(\xi) = \left( \frac{2}{1 + A^{-1}} \right)^{1/2} \xi
    \end{align*}
    Letting $\cK_\R = j(\cH_\R)$, it turns out that $\cK_\R \bigcap \iota \cK_\R = \{0\}$ and $\overline{\cK_\R + \iota \cK_\R} = \cH$. 

    The \textit{free Araki-Woods factor} associated to $(\cH_\R,U)$, as introduced in \cite{Shlyakhtenko97}  is the von Neumann algebra generated by $\{s(\xi) \; | \; \xi \in \cK_\R\}$ in $\cF(\cH) \subset \cB(\cF(\cH))$ and is denoted by $\Gamma(\cH_\R,U)''$. The vector state $\phi_U$ associated to the vacuum vector $\Omega$ given by $\phi_U(x) = \langle x\Omega,\Omega \rangle$ is called the \textit{free quasi-free state}. It is shown in \cite{Shlyakhtenko97} that for any $\xi \in \cK_\R$, the modular automorphism group is given by: 
    \begin{align*}
        \sigma^{\phi_U}_t(s(\xi)) = s(U_t(\xi))
    \end{align*}

    Thanks to the spectral theorem, given a pair $(\cH_\R,U)$, we get a Borel measure $\mu$ on $\R$ and a multiplicity function $m: \R \rightarrow \{0,1,2,...,\infty\}$ such that: 
    \begin{align*}
        \cH = \int^{\oplus}_\R \C^{m(x)} \; d\mu(x) \text{ and } U_t(\xi)(x) = e^{itx}\xi(x)
    \end{align*}
    The free Araki-Woods factor $\Gamma(\cH_\R,U)''$ is also denoted by $\Gamma(\mu,m)''$. It can be checked that $\mu$ is a symmetric measure and $m$ is a symmetric multiplicity function. Conversely given a symmetric Borel measure on $\R$ and a symmetric multiplicity function, one can reconstruct the pair $(\cH_\R,U)$. 
    \medskip
    \begin{example}
    \label{Example: Free araki woods examples}
        \begin{enumerate}
            \item If $U$ is the trivial representation on a Hilbert space $\cH_\R$ of dimension $n$, then the modular group is trivial and one gets that  $\Gamma(\id, \cH_\R)'' = L(\F_n)$. 
            
            \item If $U$ is almost periodic, the spectral measure is purely atomic. For $0<\lambda<1$, let $a = |\log \lambda|$. Let $T_\lambda$ denote the free Araki-Woods factor given by $\Gamma(\delta_a + \delta_{-a},1)''$ where $1$ denotes the constant multiplicity function. Then $T_\lambda$ is a factor of type III$_\lambda$. Its quasi-free state $\psi_\lambda$ has a periodic modular automorphism group with period $\frac{2\pi}{a}$. In fact the discrete core is isomorphic to $L(\F_\infty) \otimes \cB(\cH)$ (see \cite[Corollary 6.8]{Shlyakhtenko97}) 

            \item If $U: \R \actson L^2(\R)$ is the regular representation, then the free Araki-Woods factor $\Gamma(L^2(\R),U)''$ is of type III$_1$ by \cite[Theorem 6.10]{Shlyakhtenko97} and its continuous core is known to be $L(\F_\infty) \otimes \cB(\cH)$. Equivalently the spectral measure is equivalent to the Lebesgue measure. We shall denote this free Araki-Woods factor by $(R_{\reg}, \phi_{\reg})$. In this case the dual action is precisely the trace-scaling action constructed in \cite{Radulescu}. 
        \end{enumerate}
    \end{example}

    \subsection{Lie theory}
Most of the material below can be found in \cite{Hall_book}, \cite{Knapp_book}, or in a standard course on Lie groups and Lie algebras. We summarise a few notions which will become useful in Section \ref{Sec: Dichotomy for Lie groups}.

    Recall that a Lie group is a group $G$ endowed with the structure of a differential manifold such that both the multiplication $G\times G\to G$ and the inversion map $G\to G$ are smooth. We will only work with real Lie groups, meaning the differential structure is that of a real manifold. The right notion of map in the category of Lie groups is that of a smooth group homomorphism but fortunately a continuous homomorphism of Lie groups is automatically smooth.  
    
    If $G$ is such a group then the tangent space at the identity $T_eG$ has the structure of a finite dimensional real Lie algebra, usually denoted $\mathfrak{g}.$ That is to say $\mathfrak{g}$ is a finite dimensional real vector space endowed with a skew-symmetric bilinear map $[\cdot,\cdot]:\mathfrak{g}\times \mathfrak{g}\to \mathfrak{g}$ which satisfies the Jacobi identity: $$[X,[Y,Z]]+[Y,[Z,X]]+[Z,[X,Y]]=0 \text{ for all } X,Y,Z\in \mathfrak{g}.$$ 

    The assignment $G\to \mathfrak{g}$ is functorial meaning for every continuous homomorphism $\varphi:G\to H$ of Lie groups, the differential satisfies $$(D_e\varphi)[x,y]=[D_e\varphi(x),D_e\varphi(y)].$$ 

    Moreover, one can consider the \textit{ exponential map} $\exp:\mathfrak{g}\to G.$ Recalling by definition $\mathfrak{g}=T_eG,$ the exponential is formally given by $\exp(X_e)=\gamma(1)$ where $\gamma(t)$ is the one-parameter subgroup with tangent vector $X_e$ at the identity. This map is a local diffeomorphism but it is not in general bijective. It is also not a homomorphism unless $\mathfrak{g}$ is abelian, but rather its governed by the Baker-Campbell-Hausdorff formula: for any $X,Y\in \mathfrak{g}$ close enough to $0$ one has 
    \begin{multline*}
      \exp(X)\exp(Y)=\exp(X+Y+ \frac{1}{2}[X,Y]+\frac{1}{12}[X,[X,Y]] \\  
      +\frac{1}{12}[Y,[Y,X]]+...) 
    \end{multline*}
    where higher order terms involve higher order nested commutators. An explicit formula, due to Dynkin can be found for example in \cite{Jacobson_book} but we will not need it. Similarly, one has the Zassenhaus formula which gives 
    \begin{multline*}
    \exp(X+Y)=\exp(X)\exp(Y)\exp(-\frac{1}{2}[X,Y]) \\ \exp(\frac{1}{6}(2[Y,[X,Y]]+[X,[X,Y]]))...    
    \end{multline*}
    for sufficiently small $X,Y\in \mathfrak{g},$ where for all $n\geq 3$ the $n^{\text{th}}$ factor in the product only involves commutators of degree $n-1.$

    As a particular instance of the functoriality $G\mapsto \mathfrak{g},$ one recovers the \textit{adjoint representation}: for any element $g\in G,$ conjugation by $g$ is an automorphism of $G$ whose differential at the identity induces an automorphism $\text{Ad}(g)$ of the Lie algebra $\mathfrak{g}=T_eG$ and indeed $G\ni g\mapsto \text{Ad}(g)\in \text{Aut}(\mathfrak{g})$ is a group homomorphism. If $g$ happens to be in the image of the exponential map $g=\exp(X)$ for some $X\in \mathfrak{g}$ then Campbell's Identity ensures
    \begin{align*}
        \Ad(\exp(X))(Y)&=\exp (\text{ad}(X))(Y)=\sum_{n\geq 0}\frac{1}{n!}\text{ad}^n(X)(Y) \\ &=\sum_{n\geq 0}\frac{1}{n!}[X,[X,...,[X,Y]]]    
    \end{align*}
    where the $n^{\text{th}}$ term above contains $n$ copies of $X.$  

    Here the map $\text{ad}(X):\mathfrak{g}\to \mathfrak{g}$ is the adjoint action of $X$ on $\mathfrak{g}.$ It is a derivation, meaning it satisfies $$\text{ad}(X)\left([Y,Z] \right)=[\text{ad}(X)(Y),Z]+[Y,\text{ad}(X)(Z)] \text{ for all } X,Y,Z\in \mathfrak{g}.$$

    The map $\mathfrak{g}\ni X\to \text{ad}(X)\in \text{End}_{\mathbb{R}}(\mathfrak{g})$ provides a bilinear form $B$ -called the Killing form on $\mathfrak{g}$ given by $$B(X,Y)=\text{Tr}(\text{ad}(X)\text{ad}(Y))$$ where $\text{Tr}$ is the unnormalised trace on $\text{End}_\mathbb{R}(\mathfrak{g})\cong M_{dim_\mathbb{R}(\mathfrak{g})}(\mathbb{R}).$  

    Cartan's Criterion then says $B$ is non-degenerate if and only if $\mathfrak{g}$ is semi-simple, which is equivalent to saying the maximal solvable ideal (called radical) of $\mathfrak{g}$ is trivial. In this context an ideal $\mathfrak{i}$ of a Lie algebra $\mathfrak{g}$ is a linear subspace satisfying $[\mathfrak{i},\mathfrak{g}]\subseteq \mathfrak{i}$ and a Lie algebra is called solvable if the derived series given by $\mathfrak{g}^{(1)}=\mathfrak{g}$ and $\mathfrak{g}^{(k+1)}=[\mathfrak{g}^{(k)},\mathfrak{g}^{(k)}]$ eventually becomes trivial. Sums of solvable ideals remain solvable ensuring every finite dimensional Lie algebra has a maximal solvable ideal.   
    
    \subsection{$\R$-kernels and group cohomologies}

     We need some standard results on measurable and continuous group cohomology and for all basic definitions, we refer the reader to the seminal works of Moore (\cite{Moore_cohomology12}, \cite{Moore_cohomology3} and \cite{Moore_cohomology4}) and the work of Austin and Moore \cite{Austin_Moore}. For a locally compact second countable group $G$ and a locally compact abelian group $A$ there are two cohomology theories we will be interested in. On the one hand, we have the continuous cohomology defined by working with continuous $A$-valued $n$-cocycles/coboundaries on $G,$ denoted $Z^n_{\cts}(G,A)$ and $B^n_{\cts}(G,A)$. On the other hand, endowing $G$ with a left Haar measure $\mu_G,$ one can consider equivalence classes, modulo equality $\mu_G$-everywhere, of measurable $A$-valued $n$-cocycles/coboundaries - denoted $Z^n_m(G,A)$ and $B_m^n(G,A).$ The respective quotients yield the continuous and respectively measurable $n$-cohomology groups that we denote by $H^n_{\cts}(G,A)$ and $H^n_m(G,A)$. For what follows we will be interested in the specific case when $G = \R$ and $A = \T, \R, \Z$.   

    Let $M$ be a full factor. A continuous group homomorphism $\phi: \R \rightarrow \Out(M)$ is an $\R$-\textit{kernel}. If $\phi_0$ is a Borel lift $\R \rightarrow \Aut(M)$, then $\phi_0(s)\phi_0(t)$ and $\phi_0(s+t)$ have the same image in $\Out(M)$ and hence differ by an inner automorphism. One can choose a Borel map $\R\times \R \ni (s,t)\mapsto u_{s,t}\in \mathcal{U}(M)$ such that $\phi_0(s)\phi_0(t) = \Ad(u_{s,t}) \phi_0(s+t)$. Now we get: 
    \begin{align*}
        &\phi_0(r)(\phi_0(s)\phi_0(t)) = \phi_0(r) \Ad(u_{s,t})\phi_0(s+t) = \Ad(\phi_0(r)(u_{s,t}))\phi_0(r) \phi_0(s+t) \\ &= \Ad(\phi_0(r)(u_{s,t})) \Ad_{u_{r,s+t}}\phi_0(r+s+t) \text{ and }(\phi_0(r)\phi_0(s)) \phi_0(t)  \\ &= \Ad(u_{r,s})\phi_0(r+s)\phi_0(t) = \Ad(u_{r,s})\Ad(u_{r+s,t})\phi_0(r+s+t)
    \end{align*}
    Thus $\Ad(u_{r,s}u_{r+s,t}) = \Ad(\phi_0(r)(u_{s,t})) \Ad_{u_{r,s+t}}$ and hence there is a scalar $\Omega(r,s,t) \in \T$ such that the unitaries differ by:
    \begin{align*}
        u_{r,s}u_{r+s,t} = \Omega(r,s,t)\phi_0(r)(u_{s,t})u_{r,s+t}
    \end{align*}
    One can check that $\Omega \in Z^3_m(\R,\T)$ and the cohomology class $[\Omega] \in H^3_m(\R,\T)$ does not depend on the choice of the unitaries or the lift. This cohomology class is called the \textit{obstruction} of the kernel $\phi$ and denoted by $\Obs(\phi)$.
    We need the following lemma for our results. We suspect this will be well known to the experts but finding a reference proves to be a nontrivial task, we therefore include a proof. 
    
    \medskip
    \begin{lemma}
    \label{Lemma: trivial cohomology of R}
        For $n \geq 2$, the cohomology groups $H^n_m(\R,\T)$ are trivial.  
    \end{lemma}
    \begin{proof}
        The short exact sequence $1 \rightarrow \Z \rightarrow \R \rightarrow \T \rightarrow 1$ gives a long exact sequence of cohomology groups of the form:
        \begin{align*}
            H^i_m(\R,\R) \rightarrow H^i_m(\R,\T) \rightarrow H^{i+1}_m(\R,\Z) \rightarrow H^{i+1}_m(\R,\R)
        \end{align*}
        Since the maximal compact subgroup of $\R$ is $\{0\}$, as in the proof of \cite[Proposition 72]{Austin_Moore}, one has $H^i_m(\R,\Z) = H^i_m(\{0\},\Z)$ for all $i \geq 1$ as an application of \cite[Proposition 62]{Austin_Moore}. Thus $H^i_m(\R,\Z) = \{0\}$ for all $i \geq 1$. By \cite[Theorem A]{Austin_Moore}, $H^n_m(\R,\R) = H^n_{\cts}(\R,\R)$ for all $n \geq 2$. By the discussion in Remark \ref{Remark: Lie algebra cohomology} below, $H^n_{\cts}(\R,\R) = \{0\}$ for $n \geq 2$. Thus from the long exact sequence we have that $H^n_m(\R,\T) = \{0\}$ for all $n \geq 2$. 
    \end{proof}

    \begin{remark}
    \label{Remark: Lie algebra cohomology}
        We briefly recall the notion of Lie algebra cohomology and refer the reader to \cite{Borel_Wallach} for a more elaborate treatment. Let $\mathfrak{g}$ be a $\R$-Lie algebra and $V$ be a $\mathfrak{g}$-module. The $n$-cochains are defined by the group of alternating multilinear maps $C^n(\mg,V) = \hom_\R(\wedge^n \mg, V)$. This cochain complex comes equipped with the usual differential $d: C^n(\mg,V) \rightarrow C^{n+1}(\mg,V)$ given by: 
    \begin{align*}
        d\omega(X_1,...,X_n) = \sum_{i=0}^{n} (-1)^i X_i \cdot & \omega(X_0,...,\widehat{X_i},...X_n) \\ & +\sum_{0 \leq i < j \leq n} (-1)^{i+j} \omega([X_i,X_j],X_0,...,\widehat{X_i},...X_n)
    \end{align*}
    This sequence is usually called the Chevalley-Eilenberg sequence and the cocycle and coboundary groups are given as usual by $Z^n(\mg,V) = \ker(d)$ and $B^n(\mg,V) = \Img(d)$. The corresponding cohomology groups are denoted $H^n(\mg,V) = Z^n(\mg,V)/ B^n(\mg,V)$. One can check that when $\mg$ is abelian and $V$ has the trivial action of $\mg$, then both the terms in the definition of $d$ vanish and hence $d = 0$, and $H^n(\mg,V) = \hom(\wedge^n \mg, V)$. When $\mg = \R$ and $V = \R$ with the trivial $\R$-action, in fact $\wedge^n \mg = 0$ for $n \geq 2$ and therefore $H^n(\mg, \R) = \{0\}$ for $n \geq 2$. By the results of \cite{Hochschild_Mostow}, since $\R$ has no compact subgroups, $H^n_{\cts}(\R,\R) = H^n(\mg = \R, \R) = \{0\}$ for $n \geq 2$. 
    \end{remark}

\subsection{Prime factorization for semifinite factors}

    As in \cite[Definition 6.2]{HoudayerMarrakchiVerraedt}, we call a pair of tensor product decompositions $M = A_1 \otimes A_2 = B_1 \otimes B_2$ \textit{unitarily conjugate} and denote it by $(A_1,A_2) \sim (B_1,B_2)$ if there is a unitary $u \in \cU(M)$ such that $uA_iu^* = B_i$ for $i \in \{1,2\}$. Similarly we say that the pairs are \textit{stably unitarily conjugate} and denote by $(A_1,A_2) \sim_\infty (B_1,B_2)$ if there are type I factors $F_1$ and $F_2$ such that $(A_1 \otimes F_1,A_2 \otimes F_2)$ is unitarily conjugate with $(B_1 \otimes F_1,B_2 \otimes F_2)$ in $M \otimes F_1 \otimes F_2$. Recall from \cite[Proposition 6.3]{HoudayerMarrakchiVerraedt} that $\sim$ and $\sim_\infty$ are both equivalence relations.

    Recall from \cite{Ozawa_Popa} that a countable non-amenable ICC group $G$ is said to be in class $\mathcal{C}$ if $G$ is either a subgroup of a hyperbolic group or a discrete subgroup of a connected simple Lie group of rank 1. The following proposition will be useful for our results in Section \ref{Subsec: tensor products}.

    \medskip     
    \begin{proposition}
    \label{Prop: unique prime factors for semifinite factors}
        Let $G$ be a group in class $\mathcal{C}$ and let $N_0 = L(G) \otimes \cB(\cH)$. Let $N = N_0 \otimes N_0$ and let $B_1 = N_0 \otimes 1 \subset N$ and $B_2 = 1 \otimes N_0 \subset N$. If $N = D_1 \otimes D_2$ for II$_\infty$ factors $D_1$ and $D_2$, then there is a permutation $\pi$ of the set $\{1,2\}$ such that $(D_1,D_2) \sim (B_{\pi(1)}, B_{\pi(2)})$.
    \end{proposition}
    \begin{proof}
        Let $\tau_0$ be the semifinite trace on $N_0$ and let $\tau = \tau_0 \otimes \tau_0$ on $N$. Pick traces $\tau_1$ and $\tau_2$ on $D_1$ and $D_2$ respectively such that $\tau = \tau_0 \otimes \tau_0 = \tau_1 \otimes \tau_2$. Let $e_1$ and $e_2$ be finite projections in $D_1 \otimes 1$ and $1 \otimes D_2$ respectively such that $\tau_1(e_1) = \tau_2(e_2) = 1$. Let $p = e_1 e_2$ and notice that: 
        \begin{align*}
            \tau(p) = 1 \text{ and } pNp = pD_1p \otimes pD_2p 
        \end{align*}
        Similarly pick trace 1 projections $f_1$ and $f_2$ in $B_1$ and $B_2$ respectively such that letting $q = f_1f_2$, we have $\tau(q) = 1$ and $qNq = qB_1q \otimes qB_2q \cong L(G) \otimes L(G)$. Now choose a partial isometry $w_0$ such that $w_0^*w_0 = p$ and $w_0w_0^* = q$ as they have the same trace in $N$. Let $M_1 = w_0(pD_1p)w_0^*$ and $M_2 = w_0(pD_2p)w_0^*$. Notice that: 
        \begin{align*}
            M_1 \otimes M_2 = w_0(pD_1p \otimes pD_2p)w_0^* = w_0(pNp)w_0^*= qNq  
        \end{align*}
        By \cite[Theorem 1]{Ozawa_Popa}, there is a permutation $\pi$ of the set $\{1,2\}$ and $t>0$ such that $(M_1^t,M_2^{1/t} ) \sim ( qB_{\pi(1)}q, qB_{\pi(2)}q)$ in $qNq$. As observed in the paragraph before \cite[Theorem 1]{Ozawa_Popa}, one can check that $(M_1,M_2) \sim_\infty (M_i^t,M_2^{1/t})$ in $qNq$. Combining this with the previous unitary conjugacy, we get: 
        \begin{align*}
            ( qB_{\pi(1)}q, qB_{\pi(2)}q) \sim_\infty (M_1,M_2) \text{ in } qNq   
        \end{align*}
        From \cite[Proposition 6.3(iv)]{HoudayerMarrakchiVerraedt} applied to the two decompositions of $qNq$,  we get non-zero projection $r_i \in M_i$ and $s_i \in qB_{\pi(i)}q$, and a partial isometry $z \in qNq$ with $z^*z = r_1r_2$ and $zz^* = s_1s_2$ such that:
        \begin{align*}
            z((r_1M_1r_1)r_2)z^* &= (s_1 qB_{\pi(1)}qs_1)s_2 \\ z(r_1(r_2M_2r_2))z^* &= s_1(s_2 qB_{\pi(2)}qs_2)
        \end{align*}
        Notice that the maps $e_iD_ie_i \rightarrow M_i$ for $i \in \{1,2\}$ given by $x \mapsto w_0(xe_2)w_0^*$ and $y \mapsto w_0(e_1y)w_0^*$ are unital isomorphisms. Thus there are nonzero projections $d_i \in e_iD_ie_i$ such that $w_0(d_1e_2)w_0^* = r_1$ and $w_0(e_1d_2)w_0^* = r_2$. Since $e_2pe_1 = p$, we get: 
        \begin{align*}
            r_1r_2 = (w_0(d_1e_2)w_0^*)(w_0(e_1d_2)w_0^*) = w_0d_1pd_2w_0^* = w_0d_1d_2w_0^*  
        \end{align*}
        Similarly there are nonzero projections $g_i \in f_{\pi(i)}B_{\pi(i)}f_{\pi(i)}$ such that $s_1 = g_1f_{\pi(2)}$ and $s_2 = f_{\pi(1)}g_2$, and consequently $s_1s_2 = g_1g_2$. Now define the partial isometry $V = zw_0 \in N$ and calculate:
        \begin{align*}
            V^*V &= w_0^*z^*zw_0 = w_0^*r_1r_2w_0 = d_1d_2 \\
            VV^* &= zw_0w_0^*z^* = zqz^* = zz^* = g_1g_2
        \end{align*}
        Now we immediately get the following:
        \begin{align*}
            w_0((d_1D_1d_1)d_2)w_0^* &= (r_1M_1r_1)r_2 \implies V((d_1D_1d_1)d_2)V^* = (g_{1}B_{\pi(1)}g_1)g_2 \\ w_0 (d_1(d_2D_2d_2))w_0^* &= r_1(r_2M_2r_2) \implies V(d_1(d_2D_2d_2))V^* = g_1(g_2B_{\pi(2)}g_2)
        \end{align*}
        Since all the hypotheses of \cite[Proposition 6.3(iv)]{HoudayerMarrakchiVerraedt} are satisfied, for $N = D_1 \otimes D_2 = B_{\pi(1)} \otimes B_{\pi(2)}$, we get that $(D_1,D_2) \sim_\infty (B_{\pi(1)},B_{\pi(2)})$. As in the last sentence of \cite[Proposition 6.3(iv)]{HoudayerMarrakchiVerraedt}, since all the factors are type II$_\infty$, this automatically gives $(D_1,D_2) \sim (B_{\pi(1)}, B_{\pi(2)})$ as required.  
    \end{proof}

\section{Dichotomy for certain outer automorphism groups}\label{Sec: Dichotomy for Lie groups}

\subsection{Trace-scaling flows and $\R$-kernels}

Let $G$ be a second countable locally compact group and let $\delta: G \rightarrow \R$ be a continuous surjective homomorphism. Let us define a \textit{continuous split of $\delta$} as a continuous group homomorphism $s: \R \rightarrow G$ such that $\delta(s(t)) = t$ for all $t \in \R$. Two continuous splits $s_1$ and $s_2$ are \textit{conjugate} if there is a $g \in G$ such that $s_2(t) = gs_1(t)g^{-1}$ for all $t\in \R$. Let us denote by $\mathfrak{S}(\delta)$ the set of all continuous splits and $\mathfrak{S}(\delta)/G$ the equivalence classes of splittings under conjugacy. The following lemma follows from standard literature on Polish spaces, and we sketch a proof for completeness. 
\medskip
\begin{lemma}
    \label{Lemma: S(delta) Polish space and G action continuous}
    Let $G$ be locally compact and $\delta: G \rightarrow \R$ be a continuous surjective homomorphism. The set $\mathfrak{S}(\delta)$ is a Polish space with the topology of uniform convergence on compact sets and the conjugation action $G \actson \mathfrak{S}$ is continuous.
\end{lemma}
\begin{proof}
    Since $G$ is locally compact second countable, let $d$ be a compatible left invariant metric on $G$. Let $C(\R,G)$ be the space of continuous maps with the topology of uniform convergence of compact sets. An explicit metric is given by: 
    \begin{align*}
        D(f,h) = \sum_{n=1}^{\infty} 2^{-n} \sup_{|t| \leq n} \min \{1, d(f(t),h(t))\} 
    \end{align*}
    One checks that $C(\R,G)$ is Polish with this topology. Completeness follows from taking uniform limits on the intervals $[-n,n]$ as $n \rightarrow \infty$ and separability follows from the second countability of $\R$ and $G$. Now notice that: 
    \begin{align*}
        \mathfrak{S}(\delta) = \{s \in C(\R,G) \; | \; s(t + u) = s(t) + s(u), \delta(s(t)) = t \text{ for all } t,u \in \R\}
    \end{align*}
    Thus $\mathfrak{S}(\delta)$ is an intersection of closed sets and is closed. Since closed subsets of Polish spaces are closed, $\mathfrak{S}(\delta)$ is Polish. 

    Now we check that $G \actson \mathfrak{S}(\delta)$ is continuous. Let $g_n \rightarrow g$ in $G$ and $s_n \rightarrow s$ in $\mathfrak{S}(\delta)$, we want to show that $g_n \cdot s_n \rightarrow g \cdot s$ uniformly on compact sets. Let $K \subset \R$ be a compact subset. By local compactness, pick an open neighbourhood $V$ of $g$ such that the closure $\overline{V}$ is compact and $g_n \in V$ for sufficiently large $n$. Also $s(K)$ is compact and by local compactness there is an open set $U \subseteq G$ such that $s(U) \subseteq U$ and $\overline{U}$ is compact. Since $s_n \rightarrow s$ uniformly on $K$, a routine calculation gives that for sufficiently large $n$, $s_n(K) \subseteq U$. Thus $(g_n, s_n(t))$ and $(g,s(t))$ are both in the compact set $C = \overline{V} \times \overline{U}$ for sufficiently large $n$. 

    Let $d^{(2)}((g_1,h_1),(g_2,h_2)) = \max\{d(g_1,g_2), d(h_1,h_2)\}$ be the metric on $G \times G$. Since the conjugation map $G \times G \rightarrow G$ is uniformly continuous on the compact set $C$, for each $\epsilon > 0$, we can pick $\eta > 0$ such that: 
    \begin{align*}
        ((g_1,g_2), &(h_1,h_2)) \in C \text{ and } d^{(2)}((g_1,h_1),(g_2,h_2)) < \delta \\ & \implies d(g_1h_1g_1^{-1}, g_2h_2g_2^{-1}) < \epsilon
    \end{align*}
     Now pick $n$ sufficiently large such that we have: 
     \begin{align*}
      d(g_n,g) < \eta \text{ and } \sup_{t \in K} d(s_n(t),s(t)) < \eta   
     \end{align*}
     Then for each $t \in K$ we have $d^{(2)}((g_n,s_n(t)),(g,s(t))) < \eta$. Therefore: 
    \begin{align*}
        \sup_{t \in K} (g_ns_n(t)g_n^{-1}, gs(t)g^{-1}) \leq \epsilon
    \end{align*}
    Since $\epsilon$ is arbitrary, we get that $g_n \cdot s_n \rightarrow g \cdot s$ uniformly on $K$, as required. 
\end{proof}

Now let $M$ be a full type II$_\infty$ factor and let $\Out(M)$ be the outer automorphism group. Since each automorphism $\alpha$ satisfies $\tau \circ \alpha = \md(\alpha) \cdot \tau$, and the trace is preserved by inner automorphisms, one gets a canonical continuous group homomorphism:
\begin{align*}
  \delta = -\log (\md): \Out(M) \rightarrow \R 
\end{align*}
If $\theta: \R \actson M$ is a continuous trace-scaling flow then by defining $s_\theta: \R \rightarrow \Out(M)$ by $s_\theta(t) = [\theta_t]$, one gets: 
    \begin{align*}
        \delta (s_\theta (t)) = \delta([\theta_t]) = - \log (e^{-t}) = t
    \end{align*}
Thus $s_\theta \in \mathfrak{S}(\delta)$. We want to understand what exactly are the conjugacy classes of this set.  We remark here that continuous group homomorphisms $\R \rightarrow \Out(M)$ are also called $\R$-kernels in the literature. We shall use the terminology of $\R$-kernels and their obstructions in the proof below and we refer the interested reader to \cite{Sutherland_cohomology}.

\medskip
\begin{proposition}
\label{Prop: flows and splits}
    Let $\delta = -\log (\md): \Out(M) \rightarrow \R$ be the continuous homomorphism as above for a full II$_\infty$ factor $M$ and assume that $\cF(pMp) = \R^*_+$ for a finite corner of $M$. Then the map $\theta \mapsto s_\theta$ is a bijection between $\Flow(M)/ \sim_\cc$ and $\mathfrak{S}(\delta)/\Out(M)$. In particular when $\Flow(M)$ is non-empty, we get the bijection.  
\end{proposition}
\begin{proof}
    First, note that from the discussion in Section \ref{Subsec: full factors}, $\cF(pMp) = \R^*_+$ if and only if $\delta: \Out(M) \rightarrow \R$ is surjective, so in this case every element of $\mathfrak{S}(\delta)$ is an $\R$-kernel. In particular when $\Flow(M)$ is non-empty then any element $\theta \in \Flow(M)$ satisfies that $\delta(\theta_t) = t$ and hence $\delta$ is surjective. 
    
    Now we begin the proof: notice the map is well defined for if $\theta$ and $\eta$ are cocycle conjugate, then there exists $\rho \in \Aut(M)$ and a one parameter group of unitaries $u_t$ such that:
    \begin{align*}
        \rho \circ \theta_t \circ \rho^{-1} = \Ad(u_t) \circ \eta_t
    \end{align*}
    Taking the image in $\Out(M)$, we get $[\rho] \circ s_\theta(t) \circ [\rho]^{-1} = s_\eta(t)$ and hence $s_\eta$ and $s_\theta$ are $\Out(M)$-conjugate. For injectivity, suppose that $s_\theta$ and $s_\eta$ are $\Out(M)$-conjugate, by an element $[\rho] \in \Out(M)$. Let $\beta_t = \rho \eta_t \rho^{-1} \in \Aut(M)$. Then $[\theta_t] = [\beta_t]$ and there exists a unitary $u_t \in \cU(M)$ such that $\theta_t = \Ad(u_t) \circ \beta_t$. 
    
    Now the map $\cU(M) \rightarrow \Inn(M)$ given by $u \mapsto \Ad(u)$ is a continuous surjective Polish group homomorphism and let $\sigma: \Inn(M) \rightarrow \cU(M)$ be a Borel inverse. Since $\theta$ and $\beta$ are continuous, the map $\R \rightarrow \Inn(M)$ given by $t \mapsto \theta_t \beta_t^{-1}$ is continuous and composing it with the Borel section $\sigma$, we can in fact assume that the map $t \mapsto u_t$ is Borel. It can and will be assumed that $u_0 = 1$. Now define a Borel map $w: \R \times \R \rightarrow \T$ by $w(s,t) = u_s \beta_s(u_t)u_{s + t}^{*}$ for all $s,t \in \R$. One checks that the image of $w$ is indeed in $\T < \cU(M)$ as: 
    \begin{align*}
       \Ad(u_{s+t}) \circ \beta_{s +t} &=  \theta_{s+t} = \theta_s \circ \theta_t \\  = \Ad(u_s) \circ \beta_s \circ \Ad(u_t) \circ \beta_t &= \Ad(u_s\beta_s(u_t)) \circ \beta_{s + t}
    \end{align*}
    Thus $\Ad(u_{s+t}) = \Ad(u_s \beta_s(u_t))$ and the unitaries differ by a scalar, as required. Now we show that this Borel map is indeed a 2-cocycle. Note that: 
    \begin{align*}
        (u_r \beta_r(u_s))\beta_{r+s}(u_t) &= w(r,s)u_{r+s} \beta_{r +s}(u_t) = w(r,s)w(r+s,t)u_{r+s+t} \text{ and } \\
        u_r\beta_r(u_s\beta_s(u_t)) &= u_r\beta_r(w(s,t)u_{s+t}) = w(s,t)u_r\beta_r(u_{s+t}) \\  &= w(s,t)w(r,s+t) u_{r+s+t}
    \end{align*}
    Therefore, $w(r,s)w(r+s,t) = w(s,t)w(r,s+t)$ for all $r,s,t \in \R$ and $w$ is a 2-cocycle as required. By Lemma \ref{Lemma: trivial cohomology of R}, we have $H^2(\R,\T) = \{0\}$, and hence $w$ is a coboundary. So there is a Borel map $b: \R \rightarrow \T$ satisfying $w(r,s) = b(s)b(t)\overline{b(s+t)}$ for a.e. $s,t \in \R$. Defining $v_t \coloneqq \overline{b(t)}u_t$, one checks that $\Ad(u_t) = \Ad(v_t)$ and $v: \R \rightarrow \cU(M)$ satisfies the $\beta$-cocycle identity, i.e., $v_{s+t} = v_s \beta_{s}(v_{t})$ for a.e. $s,t \in \R$. We can, and will pick an equivalent genuine cocycle and get that $\theta_t = \Ad(v_t) \circ \beta_t = \Ad(v_t) \circ \rho \circ \eta_t \circ \rho^{-1}$ for a genuine unitary $\beta$-cocycle $v$. To deduce injectivity of the map we are only left to prove $t\mapsto v_t$ is strongly continuous. This follows from a standard Polish-space argument: Consider the crossed product $\mathcal{U}(M)\rtimes_\beta \R$ which is a Polish group and note the cocycle identity precisely shows the map $\R\ni t\mapsto (v_t,t)\in \mathcal{U}(M)\rtimes_\beta \R$ is a group homomorphism. Since $(v_t)_t$ was measurable to begin with, the automatic continuity theorem of Pettis (see for example \cite[Proposition 5]{Moore_cohomology3}) ensures the assingment $t\to (v_t,t)$ is continuous, so in particular $t\to v_t$ is itself continuous. 

    We now claim that the map is surjective. Consider an $\R$-kernel $s_0 \in \mathfrak{S}(\delta)$. Since $H^3(\R,\T)$ is trivial by Lemma \ref{Lemma: trivial cohomology of R}, the obstruction $\Obs(s_0)$ is trivial. By \cite[Theorem 4.1.3]{Sutherland_cohomology}, $s_0$ admits a Borel lift $\theta: \R \rightarrow \Aut(M)$ such that $\theta$ is a Borel homomorphism. Recall that $\Aut(M)$ is a Polish group with the u-topology. Then the Borel homomorphism $\theta$ is continuous by appealing again to the automatic continuity theorem of Pettis. Notice that $\delta([\theta_t]) = t$ and hence $\theta$ is a trace-scaling continuous flow. So $\theta \in \Flow(M)$ and $s_\theta = s_0$ and that completes the proof.
\end{proof}

We end this subsection with the following lemma that ensures that we can assume that $G$ is connected for the rest of the article. We shall denote the connected component of the identity in $G$ by $G^\circ$.  
\medskip
\begin{lemma}
\label{Lemma: connected component enough}
    Let $G$ be a Lie group and $\delta: G \rightarrow \R$ is a continuous homomorphism as above and denote by $\delta_0$ the restriction of $\delta$ to $G^\circ$. If $|\mathfrak{S}(\delta_0)/G^\circ| = \epsilon$  for $\epsilon \in \{1, 2^{\aleph_0}\}$, then  $|\mathfrak{S}(\delta)/G| = \epsilon$. 
\end{lemma}
\begin{proof}
    Let $s \in \mathfrak{S}(\delta)$, and since $s$ is continuous the image of $s$ is in $G^\circ$, so as sets $\mathfrak{S}(\delta) = \mathfrak{S}(\delta_0)$. Clearly if the $G^\circ$-orbit of $s$ is $\mathfrak{S}(\delta_0)$ then the $G$-orbit of $s$ is $\mathfrak{S}(\delta)$. Since $G$ is a second countable Lie group, $G/G^{\circ}$ is countable. Since $|\mathfrak{S}(\delta_0)/G^\circ| \leq |\mathfrak{S}(\delta)/G| \cdot |G/G^\circ|$, we get that if $|\mathfrak{S}(\delta_0)/G^\circ| = 2^{\aleph_0}$, then $|\mathfrak{S}(\delta)/G| = 2^{\aleph_0}$, as required. 
\end{proof}

\subsection{The Lie algebra computations}
\label{Subsec: lie algebra computations}

Suppose that $G$ is a real Lie group and $\delta: G \rightarrow \R$ is a continuous Lie group homomorphism. Then $\delta$ is smooth by an application of Cartan's closed subgroup theorem (see for example, \cite[Corollary 3.50]{Hall_book}). Taking the differential gives a Lie algebra homomorphism $d\delta: \mg \rightarrow \R$ that is non-trivial as $\delta$ is surjective. We will use the following notation throughout this section: 
\begin{align*}
    \mk = \ker (d \delta) \text{ and } \ma = \{Z \in \mg \; | \; d\delta(Z) = 1\}
\end{align*}
Take an element $X \in \ma$ and notice that $\mg = \mk \oplus \R \cdot X$ as $\mk$ is a codimension-1 hyperplane. The affine hyperplane $\ma$ can be written as $X + \mk$. 
\medskip
\begin{proposition}
\label{Prop: number of orbits of a}
    Let $\ma = X + \mk \subset \mg$ as above and let $\Omega: \ma \rightarrow \mathfrak{S}(\delta)$ be the map $Z \mapsto s_Z$ where $s_Z(t) = \exp(tZ)$. Then:
    \begin{enumerate}
        \item The affine hyperplane $\ma$ is $\Ad$-invariant and hence $G \actson \ma$ by $\Ad$. 
        \item The map $\Omega$ is a $G$-equivariant bijection where $G \actson \mathfrak{S}(\delta)$ by conjugation. 
    \end{enumerate}
\end{proposition}
\begin{proof}
    Let $C_g \in \Aut(G)$ denote the homomorphism $h \mapsto ghg^{-1}$. Then notice that $\delta \circ C_g = \delta$. Differentiating this at the identity, we get $d\delta \circ \Ad(g) = d\delta$. If $Y \in \ma$, then $d\delta(Y) = 1$ and hence $d \delta(\Ad(g) Y) = 1$, so $\Ad(g)Y \in \ma$, proving 1. First let $Z \in \ma$ and consider the flow $s_Z$. We have:
    \begin{align*}
        \delta(s_Z(t)) = \delta(\exp_G(tZ)) = \exp_{\R}(d\delta(tZ)) = \exp_{\R}(t) = t
    \end{align*}
    Conversely, let $s \in \mathfrak{S}(\delta)$, then there is a unique element $Z \in \mg$ such that $s(t) = \exp(tZ)$ (for example, see \cite[Page 77]{Knapp_book}). Then we get $\delta(s(t)) = \delta(\exp(tZ)) = t$, and differentiating at 0, we get the Lie algebra homomorphism $d\delta \circ ds_0 =  \id_\R$. Evaluating both sides at $1$, we get $d\delta(ds_0(1)) = d\delta(Z)  = 1$ and hence $Z \in \ma$. Finally notice that: 
    \begin{align*}
        g\Omega(Z) & g^{-1} (t) = g\exp_G(tZ)g^{-1} = C_g(\exp_G(tZ)) \\  = &\exp_{\R}(t\Ad(g)(Z)) = \Omega(\Ad(g)Z)  
    \end{align*}
    Thus $\Omega$ is $G$-equivariant, completing the proof.
\end{proof}    

Recall than an ideal $\mathfrak{i} \subset \mg$ is called \textit{characteristic} if every derivation of $\mg$ preserves $\mathfrak{i}$. 
\medskip
\begin{proposition}
\label{Prop: solvable radical nontrivial}
    Let $\mk \subset \mg$ as before. If $\mk$ is not solvable, then $|\ma/G| = 2^{\aleph_0}$. 
\end{proposition}
\begin{proof}

    By Cartan's criterion (\cite[Proposition 1.46]{Knapp_book}), one checks that since $\mk$ is not solvable, there is an element $H \in [\mk,
    \mk]$ such that the Killing form satisfies $B(H,H) \neq 0$. Let $\mr = \rad(\mk)$ be the solvable radical and let $\ms = \mk/ \mr$ be the semisimple quotient. By abusing notation let $B$ also denote the Killing form in $\ms$ and without loss of generality, suppose that $H \in \ms$. Recall that the Killing form is preserved by any automorphism of $\ms$ (\cite[Proposition 1.119]{Knapp_book}). Now for each $c > 0$, consider the element $cH$ and notice that if $cH$ and $dH$ are conjugate by an automorphism of $\ms$, then
\begin{align*}
    B(cH,cH) = B(dH,dH) \implies c^2B(H,H) = d^2 B(H,H) \implies |c| = |d|  
\end{align*}
Therefore the set of elements $\{cH \; | \; c>0\}$ in $\ms$ are pairwise non-conjugate and consequently the set of elements $\{X_0 + cH \; | \; c> 0\}$ are not conjugate by the $\Ad$-action in $X_0 + \ms$, as required. 
\end{proof}

To complete the proof of Theorem \ref{Thm: main thm dichotomy}, we are essentially left to deal with the situation when $\mk$ is solvable. Let $\mk^{(n)}$ be the $n$-th term of the derived series. The commutator subalgebra $\mk^{(2)}$ is an ideal in $\mg$ and for what follows, we denote the quotient $\mv = \mk/ \mk^{(2)}$. Recall that the affine hyperplane $\ma =  X + \mk$ where $X \in \mg$ is an element with $d\delta(X) = 1$. Recall that $\ad(X)$ preserves the ideal $\mk$, in what follows $D$ is the derivation $\ad(X)|_{\mk}$. One checks that $D$ preserves the ideal $\mk^{(2)} = [\mk,\mk]$ and hence induces a derivation $\overline{D}$ on the quotient $\mv$.  
 \medskip 
\begin{lemma}
\label{Lemma: D is not onto}
    Let $\overline{D}$ be the derivation on $\mv$ as above. If $\overline{D}$ is not onto, then $\ma$ has $2^{\aleph_0}$ $G$-orbits under the $\Ad$-action. 
\end{lemma}
\begin{proof}
    Since $\overline{D}$ is not surjective, the quotient $\mv/ \overline{D}(\mv)$ is non trivial. For each element $Y \in \mk$, let us denote by $[Y]$ its image in $\mv$. Consider now the following map: 
    \begin{align*}
        \pi: \ma \rightarrow \mv/\overline{D}(\mv) \text{ given by } \pi(X +Y) = [Y] + \overline{D}(\mv)
    \end{align*}
    We claim that $\pi$ is constant on all $G$-orbits. Suppose the claim is true, then consider $[Y_0] \notin \overline{D}(\mv)$. Now for arbitrary $c,d > 0$, if $X + cY_0$ and $X + d Y_0$ were in the same $G$-orbit, then $\pi(X+cY_0) = \pi(X + dY_0)$ implies that $(c-d) [Y_0] \in \overline{D}(\mv)$. This in turn implies $c = d$, thus we get a continuum of elements, no two of which are in the same $G$-orbit. So we are only left to prove the claim. 

    Recall the exponential map $\exp: \mg \rightarrow G$ is a local diffeomorphism and note since $G$ is connected $\exp(\mg)$ generates $G$. Let $H$ be the subgroup of $G$ generated by the sets $\{\exp(tX) \; | \; t \in \R\}$ and $\{\exp(W) \; | \; W \in \mk\}$. Recall that $\mg = \R X \oplus \mk$ and $\mk$ is an ideal. By a straightforward application of the Zassenhaus formula, we have $\exp(tX + Y) \in H$ for all $t \in \R$ and $Y \in \mk$. This implies that $H = G$. Thus it is enough to prove that $\pi$ is preserved by $\Ad(\exp(tX))$ for $t \in \R$ and by $\Ad(\exp(W))$ for $W \in \mk$. Notice that since $\mk$ is an ideal, for $W \in \mk$, we have that $\ad(W)^i (X + Y) \in \mk^{(2)}$ for all $i \geq 2$.   
    \begin{align*}
        &\Ad(\exp(W))(X+ Y) = \exp(\ad(W))(X + Y)  \\ = & \left( \id + \ad(W) +  \frac{1}{2!} \ad(W)^2 + \frac{1}{3!} \ad(W)^3 + ... \right) (X + Y) \\ = &  \left( X + Y + [W,X] \right) \; \md(\mk^{(2)}) \\
        = & \left( X + Y -D(W) \right) \; \md(\mk^{(2)})
    \end{align*}
    Thus $\pi(\Ad(\exp W)(X+Y)) = \pi(X+ Y - D(W)) = \pi(X+Y)$ as required. Similarly  for $t \in \R$ we have:
    \begin{align*}
        \exp(t\overline{D}) &= \id + t\overline{D} + \frac{t^2}{2!}\overline{D}^2 + \frac{t^3}{3!}\overline{D}^3 + ... \\ 
        \exp(t \overline{D}) - \id &= \overline{D}(t + \frac{t^2}{2!}\overline{D} + \frac{t^3}{3!}\overline{D}^2 ...)
    \end{align*}
    Thus $(\exp(t\overline{D}) - \id) ([X + Y]) = \exp(t\overline{D})([X+Y]) - [X + Y] \in \overline{D}(\mv)$. But since $\exp(t\overline{D})[X] = [X]$, in the quotient $\mv/\overline{D}(\mv)$ we in fact get $\exp(t\overline{D})([Y]) + \overline{D}(\mv) = [Y] + \overline{D}(\mv)$. Thus $\pi(\exp(\ad(tX))(X+ Y)) = \exp(\overline{tD})(\overline{Y}) + \overline{D}(\mv)$ and $\pi(X+Y) = [Y] + \overline{D}(\mv)$ as required. This completes the claim and the proof.
\end{proof}

For the next lemma, we denote by $\ma_1$ the linear subspace of $\ma$ given by $X + \mk^{(2)}$ and let $\mv = \mk/\mk^{(2)}$. Consider the quotient $q : \mg \rightarrow \mg/\mk^{(2)}$ and let $[X] = q(X)$. By abusing notation let $q$ still denote the quotient map $q: \ma \rightarrow [X] + \mv$. One checks that for $g \in G$, the adjoint map $\Ad(g)$ preserves $\ma_1$ and hence induces a map $\overline{\Ad}(g): [X] + \mv \rightarrow [X] + \mv$. Moreover $q$ is $G$-equivariant, i.e., $q \circ \Ad(g) = \overline{\Ad}(g) \circ q$. 
\medskip
\begin{lemma}
\label{Lemma: Lie algebra of H}
    Let $H = \{g \in G \; | \; \Ad(g)\ma_1 = \ma_1\} < G$, then $H$ is the stabilizer of $[X]$ under the $\overline{Ad}$ action. Assume that the derivation $\overline{D}$ on $\mv$ is onto. Then the Lie algebra $\mh$ of $H$ is equal to $\R X \oplus \mk^{(2)}$.
\end{lemma}
\begin{proof}
    First let us show that $H = \Stab([X])$. Suppose first that $\Ad(g) \ma_1 = \ma_1$. Since $X \in \ma_1$, then $\Ad(g) X \in \ma_1$ and therefore: 
    \begin{align*}
        \overline{\Ad}(g) ([X]) = q(\Ad(g)X) \in q(\ma_1) = [X] \text{ so } H \subseteq \Stab([X])
    \end{align*}
    Conversely let $g \in \Stab([X])$ and let $Z \in \ma_1$. Notice that $q(\Ad(g)Z) = \overline{\Ad}(g)(q(Z)) = \overline{\Ad}(g)([X]) = [X]$. Since $q^{-1}([X]) = \ma_1$ we have that $\Ad(g)(\ma_1) \subseteq \ma_1$. Taking $g^{-1}$ instead of $g$, we have $\Ad(g)\ma_1 = \ma_1$ as required, so $H = \Stab([X])$. 
    
    Now $\mh$ is the Lie algebra of $H$ (or equivalently the connected component at the identity $H^\circ$) and recall that:
    \begin{align*}
        \mh \coloneqq \{Z \in \mg \; | \; \exp_G(tZ) \in H \text{ for all } t\in \R\}
    \end{align*}
    Let $Z  = cX + W \in \mg$ where $c \in \R$ and $W \in \mk$, so that $[Z] = c[X] + [Y]$. Then for each $t \in \R$ and $Z \in \mg$ we have:  
    \begin{align*}
        \exp(t \ad([Z])) ([X]) = \Ad(\exp(tZ))([X]) = [X] \iff Z \in \mh
    \end{align*}
    Moreover $\exp(t \ad([Z]))([X]) = [X]$ if and only if $[[Z],[X]] = 0$. Indeed, notice that differentiating $\exp(t \ad([Z]))([X]) = [X]$ at $t = 0$ gives $[[Z],[X]] = 0$ and conversely if $\ad([X])([Z]) = 0$, then all positive powers of $\ad([Z])$ evaluated at $[X]$ is zero, implying that $\exp(t \ad([Z]))([X]) = [X]$. But we have: 
    \begin{align*}
        [[Z],[X]] = [c[X] + [Y], [X]] = [[Y],[X]] = -\overline{D}([Y])
    \end{align*}
    Therefore $Z \in \mh$ if and only if $\overline{D}([Y]) = 0$. Since $\overline{D}$ is onto, it is an invertible linear map and hence $\overline{D}([Y]) = 0 $ if and only if $[Y] = 0$. Thus $\mh = \R X \oplus \mk^{(2)}$ as required.  
\end{proof}

\subsection{Proof of the main theorem}

We now have all the necessary ingredients to prove Theorem \ref{Thm: main thm dichotomy}. Assume the notation on Lie algebras from Section \ref{Subsec: lie algebra computations}. We split the proof in two propositions: 
\medskip
\begin{proposition}
    \label{Main theorem: Lie group case}
    Let $G$ be a second countable finite dimensional real Lie group and $\delta: G \rightarrow \R$ be a continuous surjective homomorphism. Then $|\mathfrak{S}(\delta)/G| = 1$ or $2^{\aleph_0}$. 
\end{proposition}

\begin{proof}
    By Lemma \ref{Lemma: connected component enough}, we can assume that $G$ is connected. Let $\mg$ be the Lie algebra of $G$ and $\mk = \ker(d\delta)$. let $\ma = X + \mk$ and by Proposition \ref{Prop: number of orbits of a}, $|\mathfrak{S}(\delta)/G| = |\ma/G|$ where $\ma/G$ denotes the orbits of $\ma$ under the $\Ad$-action. If $\mk$ is not solvable then $|\ma/G| = 2^{\aleph^0}$ from Proposition \ref{Prop: solvable radical nontrivial}. So we can assume that $\mk$ is solvable. 
    
    We use induction on the length of the derived series of $\mk$. Suppose first that the length is 1, i.e., $\mk$ is abelian and $\mk^{(2)} = [\mk,\mk] = 0$. Let $D = \ad(X)|_{\mk}$, and if $D$ is not onto then $|\ma/G| = 2^{\aleph^0}$ by Lemma $\ref{Lemma: D is not onto}$. If $D$ is onto, let $Y \in \mk$ be arbitrary and let $W \in \mk$ be such that $D(W) = Y$. Since $W,Y \in \mk$, therefore $[W,Y] = 0$ and $[W,X] = -D(W)$. Therefore $\ad(W)(X + Y) = -D(W)$. Notice that for $i \geq 2$, $\ad(W)^k (X + Y) = 0$ since $\mk$ is abelian. Therefore: 
    \begin{align*}
        &\Ad(\exp(W))(X+Y) = \exp(\ad(W))(X+Y) \\ = X &+ Y + \ad(W)(X+Y) = X + Y -D(W)  = X
    \end{align*}
    Thus the orbit of every element of $\ma = X + \mk$ contains $X$ and hence $|\ma/G| = 1$. 

    Now suppose the result holds if $\mk$  has a derived series of length less than $n$. Let $\mk^{(2)} = [\mk,\mk]$ as before. Let $\mv = \mk / \mk^{(2)}$ and let $\overline{D}$ be the induced derivation on $\mv$. If $\overline{D}$ is not onto, then again by Lemma \ref{Lemma: D is not onto} we have $|\ma/G| = 2^{\aleph_0}$. Suppose now that $\overline{D}$ is onto and let $\ma_1 = X + \mk^{(2)} \subset \ma$. Now let $X +Y \in \ma$, we claim that the $G$-orbit of $X + Y$ intersects $\ma_1$ non-trivially. To see this let $[Y] \in \mv$ be the image of $Y$. Since $\overline{D}$ is onto, we can pick an element $W \in \mk$ such that $\overline{D}([W]) = [Y]$. Similar to the abelian case, we get: 
    \begin{align*}
        \Ad(\exp(W))(X + Y) = X + Y - D(W) = X \; \; (\md \;  \mk^{(2)})
    \end{align*}
    Therefore $\Ad(\exp(W)) \in X + \mk^{(2)} = \ma_1$, proving the claim. Let $H = \{h \in G \; | \; \Ad(h)\ma_1 = \ma_1 \}$ and by Lemma \ref{Lemma: Lie algebra of H}, the Lie algebra $\mh$ is $\R X \oplus \mk^{(2)}$.
    
    Also by Lemma \ref{Lemma: Lie algebra of H}, two elements of $\ma_1$ are $G$-conjugate if and only if they are $H$-conjugate. Indeed notice that if $\Ad(g)Z_1 = Z_2$ for $g \in G$ and $Z_1,Z_2 \in \ma_1$, then applying the quotient map $q$, we get: 
    \begin{align*}
      \overline{\Ad}(g)([X]) =\overline{\Ad}(g)(q(Z_1)) =q(\Ad(g)Z_1) = q(Z_2) = [X]
    \end{align*}
    Thus $g \in H$ as required. By the previous claim, every $G$-orbit in $\ma$ intersects $\ma_1$ and this in fact implies that $|\ma/G| = |\ma_1/ H|$. Letting $H^\circ$ be the connected component of the identity in $H$, we have that $|\ma_1/H^\circ| = 1 $ or $2^{\aleph_0}$ by the induction hypothesis. Thus by Lemma \ref{Lemma: connected component enough}, we have $|\ma_1/H| = |\ma/G| = 1$ or $2^{\aleph_0}$, thus completing the proof.   
\end{proof}

\begin{proof}[Proof of Theorem \ref{Main theorem: general lc group case}]
    By Lemma \ref{Lemma: S(delta) Polish space and G action continuous}, $\mathfrak{S}(\delta)$ is a Polish space and $G \actson \mathfrak{S}(\delta)$ is continuous. By \cite[Lemma 1.3]{HofmannWuYang}, $\mathfrak{S}(\delta)$ is non-empty. For contradiction, assume that $|\mathfrak{S}(\delta)/G| < 2^{\aleph_0}$. By the usual Glimm-Effros dichotomy theorem (see \cite[Theorem 3.4.2]{Becker_Kechris}), the orbit equivalence relation on $\mathfrak{S}(\delta)$ is smooth and admits a Borel selector. 

    Now use the Gleason-Yamabe theorem to pick an open subgroup $H<G$ and a sequence of compact normal groups $K_m < H$ such that $\bigcup_m K_m = \{e\}$ and $H/K_m$ is a Lie group for all $m$. Since a continuous surjective homomorphism between Polish groups is automatically an open map, $\delta(H)$ is open in $\R$ and hence $\delta(H) = \R$. Since $K_m$ is compact, $\delta(K_m) = \{e\}$ for each $m$. Thus we obtain (still surjective continuous) group homomorphisms $\delta_m : H/K_m \rightarrow \R$. Since every open subgroup of a locally compact group is closed, $H<G$ is closed and hence once again by \cite[Theorem 3.4.2]{Becker_Kechris}, the orbit equivalence relation of $H \actson \mathfrak{S}(\delta)$ is smooth and every orbit of $H \actson \mathfrak{S}(\delta)$ is a $G_\delta$-subset of $\mathfrak{S}(\delta)$. 
    
    If $H \actson \mathfrak{S}(\delta)$, then of course so does $G \actson \mathfrak{S}(\delta)$. Suppose that $H \actson \mathfrak{S}(\delta)$ has at least two orbits. By the previous paragraph, both orbits are $G_\delta$-subsets. Since orbits are disjoint, they cannot be both dense by Baire category theorem. `Let $s \in \mathfrak{S}(\delta)$ be such that $H \cdot s$ is not dense and let $s' \in \mathfrak{S}(\delta)$ be such that $s' \notin \overline{H \cdot s}$. Since $H<G$ is open and since the only open subgroup of $\R$ is $\R$, we have that $s(\R) = s'(\R) \subset H$. Thus it follows that $s$ and $s'$ are both in $\mathfrak{S}(\delta_m)$ for all $m$. We claim now that there is a positive integer $m_0$ such that $s'_{m_0} \notin (H/K_{m_0}) \cdot s_{m_0}$. Suppose not, then for each $m \in \N$ we find $h_m \in H$ such that: 
    \begin{align*}
        h_m s(t) h_m^{-1} = s'(t) \; \md(K_m) \text{ for all } t\in \R
    \end{align*}
    Therefore $h_m \cdot s \rightarrow s'$ in $\mathfrak{S}(\delta)$ which contradicts our assumption, thus proving the claim. Therefore $s_{m_0}' \notin (H/K_{m_0}) \cdot s_{m_0}$ and by Proposition \ref{Main theorem: Lie group case}, we get that $|\mathfrak{S}(\delta_{m_0})/ (H/K_{m_0}) | = 2^{\aleph_0}$. Notice that the map $\mathfrak{S}(\delta) \rightarrow \mathfrak{S}(\delta_{m_0})$ is $H$-equivariant, and since these are Polish spaces the map is surjective (again by \cite[Lemma 1.3]{HofmannWuYang}). Thus we have $|\mathfrak{S}(\delta)/H| = 2^{\aleph_0}$. Since $H<G$ is open, by second countability we have that $G/H$ is countable and hence $|\mathfrak{S}(\delta)/G| = 2^{\aleph_0}$.
\end{proof}

\begin{proof}[Proof of Theorem \ref{Thm: main thm dichotomy}]
       Let $G = \Out(M)$ be a locally compact second countable group and let $\delta = - \log (\md): G \rightarrow \R$. Since $\cF(pMp) = \R^*_+$, $\delta$ is surjective.  By Proposition \ref{Prop: flows and splits}, it is enough to prove that $|\mathfrak{S}(\delta)/G| = 1$ or $2^{\aleph_0}$. This now follows from Theorem \ref{Main theorem: general lc group case}. 
\end{proof}

\subsection{Applications}

We give two applications of the results of this section that illustrate the two conclusions of the dichotomy. Recall from \cite{Chakraborty25} the construction of a full III$_1$ factor $M_1 = L(\cR_1)$ with $\Out(M_1) = \R^*_+$. Let $M_C = L(\cR_C)$ be the continuous core (and Maharam extension) respectively. Following the proofs in \cite{Chakraborty25}, one can easily see that $\Out(M_C)$ is also $\R^*_+$. We give a short proof here. Since the construction is rather tedious, we refer the reader to \cite[Constructions 3.1 and 3.2]{Chakraborty25} and do not recall the details here.
\medskip
\begin{proposition}
\label{Prop: Chakraborty construction}
    Let $M_C = L(\cR_C)$ denote the full II$_\infty$ factor constructed in \cite{Chakraborty25}. Then $M_C$ has a unique continuous trace-scaling flow up to cocycle conjugacy. 
\end{proposition}
\begin{proof}
    Let $\cR$ be the orbit equivalence relation of the infinite measure preserving action $\Gamma \actson (X \times Y, \eta)$ from \cite[Construction 3.1]{Chakraborty25}. Exactly in a way similar to the computations of the outer automorphism group of the type III$_1$ factor whose core is $M_C$ in \cite{Chakraborty25}, one can compute $\Out(M_C)$. Indeed, by an application of \cite[Theorem 8.1]{Vaes14} (c.f. \cite[Theorem 5.6]{Chakraborty25}), we get that $M_C$ has a unique Cartan subalgebra. Moreover by \cite[Lemma 4.4]{Chakraborty25}, the equivalence relation $\cR$ is $\cU$-fin cocycle superrigid and since the acting group $\Gamma$ is perfect, the cohomology group $H^1(\cR,\T) = \{e\}$. Therefore by the usual Feldman-Moore correspondence of automorphism groups, we get $\Out(M_C) = \Out(\cR_C)$. Let $\theta \in \Aut(\cR_C)$, then by ergodicity of $\cR_C$, we get a constant $\md(\theta) \in \R^*_+$ such that $\theta_*\eta = \md(\theta) \cdot \eta$. Thus we get a homomorphism $\md: \Aut(\cR_C) \rightarrow \R^*_+$. If $\theta \in [\cR_C]$ then it preserves the measure $\eta$ and hence $[\cR_C] \subseteq \ker(\md)$. 

    Now suppose $\md(\theta) = 1$ for some $\theta \in \Aut(\cR_C)$. Then by \cite[Proposition 5.4]{Chakraborty25}, there is an element $\phi \in [\cR]$ such that $\phi \circ \theta$ commutes with the $\Gamma$-action. Since $\phi$ and $\theta$ are both measure preserving, by \cite[Proposition 5.5]{Chakraborty25}, $\phi \circ \theta$ is a dilation $D_s$ given by $D_s(x,y) = (x,sy)$ for some $s \neq 0$. Since $D_s$ needs to be measure preserving, we have that $s = \pm 1$. Exactly as in the proof of \cite[Proposition 5.9]{Chakraborty25}, $D_{-1} \in [\cR_C]$. Thus $\phi \circ \theta \in [\cR_C]$ and hence $\theta \in [\cR_C]$, so $[\cR_C] = \ker(\md)$. Moreover one checks that $\md(D_s) = s^{-6}$ and $s$ ranges over $\R^*$, we get that $\md: \Aut(\cR_C) \rightarrow \cR^*_+$ is onto. Thus by the first isomorphism theorem we have $ \Out(\cR_C) = \Aut(\cR_C)/ [\cR_C] \cong \R^*_+$ as required. Now $\delta: \Out(M_C) \rightarrow \R$ in this case is an isomorphism of topological groups and this immediately gives that a splitting of $\delta$ must be given by $\delta^{-1}$, so $\mathfrak{S}(\delta)$ has a unique element. So by Proposition \ref{Prop: flows and splits} we get that up to cocycle conjugacy there is a unique trace-scaling continuous flow.  
\end{proof}

Now we shall see a much richer application of Theorem \ref{Thm: main thm dichotomy}. The following construction appears in \cite{Deprez12}. Let $n \geq 1$ and let $m = 4n + 1$. Let $\Lambda = \SL(m,\Z)$, $Y = M_{m,n}(\R)$ and $\nu$ be the Lebesgue measure on $M_{m,n}(\R)$ from its bijection with $\R^{mn}$. Consider the left multiplication action $\Lambda \actson (Y,\nu)$. One checks that this action is essentially free, ergodic and preserves the infinite measure $\nu$. In what follows we shall denote by $\Aut_{\MP}(\Lambda \actson Y)$ and $\Aut_{\ns}(\Lambda \actson Y)$ the Polish groups of measure-preserving and non singular automorphisms commuting with the action respectively. 
\medskip
\begin{proposition}
\label{Prop: Deprez construction}
    For each $n \geq 1$, there is a full II$_\infty$ factor $M_n$ with $\Out(M_n) = \GL(n,\R)$. For $n \geq 2$, $|\Flow(M_n)/\sim_{\cc} | = 2^{\aleph_0}$ and for $n = 1$, we have $|\Flow(M_n)/\sim_{\cc}| = 1$.     
\end{proposition}
\begin{proof}
    Let $\Lambda \actson Y$ as above and let $M_n$ be the II$_\infty$ factors constructed in \cite[Theorem 8.4]{Deprez12} for $n \geq 1$. Let $p$ be a finite projection in $M_n$ and let $B_n = pM_n p$ be a finite corner. As observed in \cite[Theorem 8.4]{Deprez12}, by cocycle superrigidity of the action and the fact that $\SL(n,\Z)$ is a perfect group, $H^1(\Lambda \actson Y, \T) = \{e\}$. For each $A \in \GL(n,\R)$, let $R_A$ denote the right multiplication $R_A(y) = yA$. Since $R_A$ is invertible (with inverse $R_{A^{-1}}$), we get that $R_A$ is non-singular. Since $\Lambda \actson Y$ is by left multiplication, each $R_A$ commutes with $\Lambda \actson Y$ and so $R_A \in \Aut_{\ns}(\Lambda \actson Y)$. Notice that the Jacobian is given by $|\det(R_A)| = |\det(A)|^m$ and hence for any Borel subset $E \subset Y$, we get $\nu(R_A(E)) = |\det(A)|^m \nu(E)$. Thus $R_A$ is in the centralizer of $\Lambda$ in $\Aut_{\ns}(\Lambda \actson Y)$ for all $A \in \GL(n,\R)$. Now for $\lambda> 0$, taking $A_{\lambda} = \lambda I$, we get that $\md(R_{A_\lambda}) = \lambda^{mn}$. Thus we have $\md(\{R_{A_\lambda} \; | \; \lambda > 0\}) = \R^*_+$. Hence by \cite[Theorem C]{Deprez12}, we get $\cF(B_n) = \R^*_+$. Moreover as observed in \cite[Theorem 8.4]{Deprez12}, $\Out(B_n) = \SL^{\pm}(n,\R) = \{R_A \; | \; A \in \SL^{\pm}(n,\R)\}$. Now consider the standard exact sequence: 
    \begin{align*}
        1 \rightarrow \Out(B_n) \rightarrow \Out(M_n) \rightarrow \cF(B_n) \rightarrow 1
    \end{align*}
    which in our case becomes: 
    \begin{align}
    \label{eq: short exact seq}
        1 \rightarrow \SL^{\pm}(n,\R) \rightarrow \Out(M_n) \rightarrow \R^*_+ \rightarrow 1
    \end{align}
    Letting $s: \R^*_+ \rightarrow \Out(M_N)$ denote the map $s(t) = [R_{A_{t^{1/mn}}}]$, notice that $\md(s(t)) = (t^{1/mn})^{mn} = t$ and $s$ is clearly continuous, so Equation \ref{eq: short exact seq} splits. In particular letting $\delta = - \log (\md)$, we get that $\mathfrak{S}(\delta)$ is non-empty. Moreover notice that $s(t)R_As(t)^{-1} = R_A$ for all $A \in \SL^{\pm}(n,\R)$ and hence we get a homeomorphism $\Out(M_n) \cong \SL^{\pm}(n,\R) \times \R^*_+$. Since $\SL^{\pm}(n,\R) \times \R^*_+$ is homeomorphic to $\GL(n,\R)$, we have that $\Out(M_n) \cong \GL(n,\R)$ is a Lie group. 
    
    Since the fundamental group of $B_n$ is $\R^*_+$ and $\mathfrak{S}(\delta)$ is non-empty, by Proposition \ref{Prop: flows and splits}, we get that $\Flow(M_n)$ is non-empty, i.e., $M_n$ admits at least one continuous trace-scaling flow. For $n \geq 2$, notice that $K = \ker(\delta) = \SL^{\pm}(n,\R)$ and the connected component of the identity is $K^\circ = \SL(n,\R)$. Since $\SL(n,\R)$ is semisimple, its Lie algebra has trivial solvable radical. Thus by Theorem \ref{Thm: main thm dichotomy}, we have that $|\Flow(M_n)/ \sim_{\cc}| = 2^{\aleph_0}$. When $n = 1$, $\Out(M_1) = \{\pm 1\} \times \R^*_+ \cong \R^*$ and $K = K^\circ = \{e \}$. One checks that $\delta$ is given by $\delta(\epsilon, t) = - \log(t)$ where $\epsilon \in \{\pm 1\}$ and $t \in \R^*_+$. Notice that the image of a continuous homomorphism $s: \R^*_+ \rightarrow \{\pm 1 \} \times \R^*_+$ has its image in $\R^*_+$ by connectedness. Thus there is exactly one split $s \in \mathfrak{S}(\delta)$ given by $s(t) = (1,e^{-t})$ and by Proposition $\ref{Prop: flows and splits}$, $|\Flow(M_1) / \sim_{\cc}| = 1$, as required.    
\end{proof}    
    
\section{Other examples of pairwise distinct trace-scaling flows}
\label{Sec: other examples}

In this section we give two different constructions of a continuum of pairwise non-cocycle conjugate continuous trace-scaling flows on full II$_\infty$ factors.

\subsection{Flows on certain tensor products}
\label{Subsec: tensor products}
    Recall that for a strongly continuous flow $\R \actson M$ on a semifinite factor, there is a unique continuous homomorphism $\rho: \R \rightarrow \R^*_+$ such that $\tau_M \circ \alpha_t = \rho(t)\tau_M$. In this case $\rho(t) = e^{\md(\alpha)t}$ for a unique real number $\md(\alpha)$ which is called the \textit{Connes-Takesaki module of the flow}. We make the following adaptation of the definition. 
    \medskip
    \begin{definition}
        Let $\alpha: \R \actson M$ is a strongly continuous flow on a semifinite factor. Let $M = A_1 \otimes A_2$ for semifinite factors $A_1$ and $A_2$ such that both $A_1$ and $A_2$ are $\alpha$-invariant. If $\tau_{A_1}$ is a faithful normal semifinite trace on $A_1$, then there is a unique continuous homomorphism $\rho: \R \rightarrow \R^*_+$ such that $\tau_{A_1} \circ \beta_t = \rho(t) \tau_{A_1}$. Thus $\rho$ is given by $\rho(t) = e^{-\md_\alpha(A_1)t}$ for a unique real number $\md_\alpha(A_1)$. We call this the \textit{partial module of $\alpha$ with respect to $A_1$}. 
    \end{definition}
    \medskip
    \begin{lemma}
    \label{Lemma:unitaries preserve the module}
        Let $\alpha: \R \actson M$ is a strongly continuous flow on a semifinite factor. Suppose $M = A_1 \otimes A_2 = B_1 \otimes B_2$ be two decompositions into semifinite factors where each factor is $\alpha$-invariant. Suppose that $uA_1u^* = B_1$ for a unitary $u \in M$, then $\md_\alpha(A_1) = \md_\alpha(B_1)$. 
    \end{lemma}
    \begin{proof}
        Let $w_t = u^* \alpha_t(u)$ and notice that since the factors are $\alpha$-invariant and since $\Ad(u)(A_1) = B_1$, we get that $A_1$ is $\Ad(w_t)$ invariant. Since $A_1$ and $A_2$ are factors, the automorphism $\Ad(w_t) \in \Aut(M)$ preserves the relative commutant of $A_1$ which is $A_2$. Thus $\Ad(w_t) = \delta_t \otimes \eta_t$ where $\delta_t \in \Aut(A_1)$ and $\eta_t \in \Aut(A_2)$. Since $\Ad(w_t)$ is inner in $M$, this implies that $\eta_t$ and $\delta_t$ are both inner in $A_1$ and $A_2$ respectively by \cite[Corollary 1.14]{Kallman}. 
        
        Thus $\Ad(w_t)$ preserves all faithful normal semifinite traces on $A_1$. Therefore we get $\md(\delta) = \md_{\alpha}(A_1)$. On the other hand $\Ad(u): A_1 \rightarrow B_1$ satisfies $\tau_{B_1} \circ \Ad(u) = c\tau(A_1)$ for some $c > 0$. Then we calculate:
        \begin{align*}
            c\tau_{A_1}(\delta_t(x)) &= \tau_{B_1} (\Ad(u) \circ \delta_t(x)) = \tau_{B_1}(\beta_t \circ \Ad(u)(x)) \\ &= e^{-\md_{\alpha}(B_1)t} \tau_{B_1}(\Ad(u)x) = c e^{-\md_{\alpha}(B_1)t} \tau_{A_1}(x) 
        \end{align*}
        Thus we also get $\md_\alpha(B_1) = \md(\delta)$, completing the proof. 
    \end{proof}
    \begin{proof}[Proof of Theorem \ref{Thm: main thm tensor case}]
        Let $G$ be in class $\mathcal{C}$ and let $\tau$ be the semifinite trace on $N_0 = L(G) \otimes \cB(\cH)$. By hypotheses, let $\beta: \R \actson N$ such that $\tau(\beta_t(x)) = e^{-t}\tau(x)$ for all $x \in N$ and $t \in \R$. Let $N = N_0 \otimes N_0 \cong L(G \times G) \otimes \cB(\cH)$. We know that amplifications of full factors are full, and tensor products of full factors are full by \cite{HoudayerMarrakchiVerraedt}. Since $G$ is in class $\mathcal{C}$, $L(G)$ is full and hence $N$ is full. Now for $s \in \R$, define the following flows: 
    \begin{align*}
        \alpha^s: \R \actson N_0 \otimes N_0, \;\; \alpha^s_t = \beta_{st} \otimes \beta_{(1-s)t}.  
    \end{align*}
    For each $s$, the flow $\alpha^s$ clearly remains continuous. Notice that they are trace-scaling as well because:
    \begin{align*}
       (\tau \otimes \tau)( \alpha^s_t(x\otimes y)) &= \tau(\beta_{st}(x)) \otimes \tau(\beta_{(1-s)t}(y))\\ &= e^{-st}\tau(x) \otimes e^{-(1-s)t}\tau(y) \\ &= e^{-t}(\tau \otimes \tau) (x \otimes y)
    \end{align*}
    For each $s \in \R$, let $M_s$ be the fixed point subalgebra $N^{\alpha^s}$ of the flow $\alpha^s$. By the continuous decomposition theorem of Takesaki, the continuous core of $M_s$ is isomorphic to $N$ for each $s$. Since $N$ is a type II$_\infty$ factor, $M_s$ is a type III$_1$ factor for all $s \in \R$. By \cite[Corollary C]{Marrakchi}, $M_s$ is a full factor for all $s \in \R$. 

    We claim that $M_s \cong M_r$ if and only if $s = r$ or $s = 1-r$. For the \textit{only if} direction, suppose that $M_s \cong M_r$. By the usual Takesaki duality, the canonical extension of this isomorphism gives an isomorphism $\Phi \in \Aut(N)$ that intertwines the canonical dual flows. Thus $\Phi$ satisfies, $\Phi \circ \alpha^s = \alpha^r \circ \Phi$. By abusing notation, let $N_1 \coloneqq N_0$ and $N_2 \coloneqq N_0$ be the canonical tensor factors in $N = N_1 \otimes N_2$. For the flow $\alpha^s$ we know that $N_1$ and $N_2$ are invariant. Notice that $\md_{\alpha^s}(N_1) = s$ and $\md_{\alpha^s}(N_1) = 1-s$. Now let $D_1 = \Phi(N_1)$ and $D_2 = \Phi(N_2)$ so that we have $N = D_1 \otimes D_2$. Since $\Phi$ is $\R$-equivariant, this gives $\md_{\alpha^r}(D_1) = s$ and $\md_{\alpha^r}(D_2) = 1-s$. 

    Now since $N = N_1 \otimes N_2 = D_1 \otimes D_2$, by Proposition \ref{Prop: unique prime factors for semifinite factors}, there is a unitary $u \in \cU(N)$ such that $uD_1 u^* = N_1$ or $uD_1 u^* = N_2$. By Lemma \ref{Lemma:unitaries preserve the module}, $\md_{\alpha^r}(D_1) = \md_{\alpha^r}(N_1) = r$ or $\md_{\alpha^r}(D_1) = \md_{\alpha^r}(N_2) = 1-r$. Thus either $s = r$ or $s = 1-r$, proving the \textit{only if} direction. The only thing left to show in the \textit{if} direction is $M_s \cong M_{1-s}$, and this follows easily from the equivariant isomorphism between $\alpha^s$ and $\alpha^{1-s}$ given by the flip map $N_0 \otimes N_0 \rightarrow N_0 \otimes N_0$.  
    \end{proof}

\section{Flows on a prime semifinite factor}
\label{Subsec: prime factors}
\begin{definition}
\label{Def: free Araki woods construction}
    Let $\nu$ be a finite symmetric Borel measure on $\R$ which is equivalent to the Lebesgue measure. For $a>0$, set $\mu[a] = \nu + \delta_a + \delta_{-a}$. Let $(M_a,\phi_a)$ be the free Araki-Woods factor $\Gamma(\mu[a],1)''$, where 1 is the constant multiplicity function, together with its usual free quasi-free state. By \cite[Theorem 2.11]{Shlyakhtenko97}, we get that $(M_a, \phi_a) \cong (T_\lambda, \psi_\lambda) \ast (R_{\reg}, \phi_{\reg})$ (following the notation of Example \ref{Example: Free araki woods examples}). 
\end{definition}

Every Borel probability measure $\eta$ on $\R$ can be written uniquely as $\eta_{\at} + \eta_{\ct}$ where $\eta_{\ct}$ is the continuous part and $\eta_{\at}$ is the atomic part. Recall that a symmetric Borel probability measure $\eta$ on $\R$ is said to be in the class $\cS(\R)$ introduced by Houdayer, Shlyakhtenko and Vaes in \cite{HouShlyaVaes} if $\eta_{\ct} \ast \eta_{\ct} < \eta_{\ct}$ and $\eta_{\at} \neq 0$ and $\supp(\eta_{\at}) \neq \{0\}$.   
\medskip
\begin{lemma}
\label{Lemma: non isomorphic factors}
    The factors $(M_a)_{a>0}$ from Definition \ref{Def: free Araki woods construction} are all type III$_1$ full factors and are pairwise non-isomorphic. 
\end{lemma}
\begin{proof}
    Recall that $M_a \cong T_\lambda \ast R_{\reg}$ for $\lambda = e^{-a}$ and by the results in \cite{Shlyakhtenko97}, we know that $T_\lambda$ is a type III$_\lambda$ full factor and $R_{\reg}$ is a type III$_1$ full factor. Since $T_\lambda$ and $R_{\reg}$ are both diffuse, $M_a$ is a full factor which is not of type III$_0$ by \cite[Theorems 3.4 and 4.1]{Ueda11}. One can easily compute the $T$-invariant by \cite[Theorem 3.4]{Ueda11} and see that $T(M_a) = \{0\}$. Since $M_a$ is not of type III$_0$, this forces $M_a$ to be of type III$_1$. 

    We first show that $\mu[a] \in \cS(\R)$. Indeed, the continuous part of $\mu[a]$ is $\nu$ which is equivalent to the Lebesgue measure $\lambda \coloneqq \lambda_\R$. Let us write $d\nu = f d\lambda$ for a function $f \in L^1(\R)_+$. Then $d (\nu \ast \nu) = (f \ast f) \; d \lambda$. Thus if $\lambda(E) = 0$ then $\int_{E} (f \ast f)(x) \; d\lambda(x) = 0$ and hence $\nu \ast \nu \preccurlyeq \lambda \sim \nu$. Therefore $\mu[a] \in \cS(\R)$. The countable subgroup $\Lambda((\mu[a])_{\at}) < \R$ generated by the atoms is $\Lambda((\mu[a])_{\at}) = a\Z$. Let $\delta_{a\Z} = \sum_{n \in \Z} c_n \delta(na)$ with $c_n > 0$ and $\sum_{n} c_n < \infty$, then one checks that $\nu_* \delta_{a\Z} \sim \lambda$. Thus by the classification result \cite[Corollary B]{HouShlyaVaes}, we get that:
    \begin{align*}
        M_a \cong M_b \iff a\Z = b\Z \iff a = b
    \end{align*}
    Thus the family $(M_a)_{a > 0}$ is a pairwise non-isomorphic family of full III$_1$ factors as required.
\end{proof}

We need the following technical lemma to complete the proof of Theorem \ref{Thm: main thm prime case}.
\medskip

\begin{lemma}
\label{Lemma: lemma for connes embeddibility}
    Let $c(T_\lambda)$ be the continuous core of $T_\lambda$. Then there is a trace one projection $p$ of full central support such that $p c(T_\lambda) p \cong L(\F_\infty) \otimes L^\infty(\T)$.  
\end{lemma}
\begin{proof}
    Since the modular automorphism group $\sigma^{\psi_\lambda}$ is periodic with period $S = - \frac{2\pi}{\log \lambda}$, the action factors through $K = \R/S\Z$. Then $D_\lambda = T_\lambda \rtimes_{\sigma^{\psi_\lambda}} K$ is the discrete core and by \cite[Corollary 6.8]{Shlyakhtenko97}, $D_\lambda \cong L(\F_\infty) \otimes \cB(\cH)$. Let $\{u_t\}_{t \in \R}$ be the canonical unitaries in $c(T_\lambda)$. Since $\Ad(u_S)$ is trivial, $u_S \in Z(c(T_\lambda))$. We claim that $c(T_\lambda)\cong D_\lambda \otimes L^\infty(\T)$. 

    Now let $h$ be a bounded self adjoint operator in $c(T_\lambda)$ such that $e^{i S h} = u_S$ (one can prove the existence of such an element $h$ for example by using Borel functional calculus). Let $v_t = e^{-ith} u_t$ for each $t \in \R$ and notice that $v$ is a continuous one-parameter group of unitaries. Notice that:
    \begin{align*}
        v_S = e^{-iSh}\lambda_S = u_S^*u_S = 1
    \end{align*}
    Hence $v_t$ factors through $K = \R/S\Z$. We denote the unitaries still by $v_k$ for $k \in K$ by abusing notation. Let $D \subset c(T_\lambda)$ be the von Neumann algebra generated by $T_\lambda$ and the unitaries $\{v_{k} \; | \; k \in K\}$. It is easy to see that $\Ad(v_t)(x) = \sigma_t^{\psi_\lambda}(x)$ and therefore $D \cong T_\lambda \rtimes_{\Ad (v_t)} K \cong D_\lambda$. 

    Now let $A = W^*(\{u_S\}) \cong L^\infty(\T)$ and consider the multiplication map $D \otimes A \rightarrow c(T_\lambda)$. This is a normal *-homomorphism and it is surjective because  
    $e^{ith} \in W^*(\{u_S\})$. The kernel of this map lies in $Z(D \times A) = Z(A) = A$. However, the map is injective on $A$ and hence the kernel is trivial. Thus we get an isomorphism and the claim is proven. 

    Now let $m$ denote the normalizer Haar measure on $\T$.and let $\tau_{D_\lambda}$ and $\tau(c(T_\lambda))$ denote the canonical semifinite trace on $D_\lambda$. Let $\Tr$ be the canonical semifinite trace on the continuous core $c(T_\lambda)$ (that is scaled by the dual flow). For convenience we write an element $x \in c(T_\lambda)$ as a function from $c: \T \rightarrow D_\lambda$ using the central decomposition. By the isomorphism we get the disintegration: 
    \begin{align*}
        \Tr(x) = \int_{\T}f(z)\tau_{D_\lambda}(x(z)) \; dm(z) \text{ for a measurable } f \geq 0 \text{ a.e.}
    \end{align*}
    Fix $e \in D_\lambda$ with $\tau_{D_\lambda}(e) = 1$ and let $\rho$ be a trace scaling flow on $D_\lambda$. Let $p(z) = \rho_{\log(f(z))}(e)$ and notice that $p$ defines an element of $c(T_\lambda)$. We calculate: 
    \begin{align*}
        \Tr(p) &= \int_{\T}f(z)\tau_{D_\lambda}(p(z)) \; dm(z) = \int_{\T}f(z)\tau_{D_\lambda}(\rho_{\log(f(z))}(e)) \; dm(z) \\ & = \int_{\T}f(z)e^{-\log(f(z))}\tau_{D_\lambda}(e) \; dm(z) =  \int_{\T} f(z)f(z)^{-1} \; dm(z) =  1
    \end{align*}
    Thus $p$ has full central support and we get $pc(T_\lambda)p = eD_\lambda e \otimes L^\infty(\T)$. Since $e D_\lambda e$ is an amplification of $L(\F_\infty)$ and the fundamental group of $L(\F_\infty)$ is $\R^*_+$, we get that $e D_\lambda e \cong L(\F_\infty)$, thus concluding the proof. 
\end{proof}

\begin{proof}[Proof of Theorem \ref{Thm: main thm prime case}]
    Let $c(M_a)$ be the continuous core of $M_a$. Since $\sigma^{\phi_a}$ restricts to $\sigma^{\psi_\lambda}$ on $T_\lambda$ and $\sigma^{\phi_{\reg}}$ on $R_{\reg}$, one gets as in the proof of \cite[Theorem 5.3]{Shlyakhtenko04} that 
    \begin{align*}
        c(M_a) \cong c(T_\lambda) \ast_{L(\R)} c(R_{\reg})
    \end{align*}
    where the amalgam $L(\R)$ is the canonical copy of $\R$ inside the crossed product. As an application of Shlyakhtenko's $A$-valued semicircular systems for $A = L(\R)$ we obtain by \cite[Theorem 5.1]{Shlyakhtenko04}: 
    \begin{align*}
        c(R_{\reg}) \cong \Phi(L(\R),\Tr) \cong L(\F_\infty) \otimes \cB(\cH) 
    \end{align*}
    In fact if $p$ is a projection in $c(T_\lambda)$ of trace 1, we get exactly as in the proof of \cite[Theorem 5.3]{Shlyakhtenko04}:
    \begin{align*}
        c(T_\lambda) \ast_{L(\R)} \Phi(L(\R),\Tr) \cong \Phi(c(T_\lambda), \Tr) \cong (pc(T_\lambda)p \ast L(\F_\infty) \otimes \cB(\cH))
    \end{align*}
    Now notice that by Lemma \ref{Lemma: lemma for connes embeddibility}, we get a projection $p_\lambda$ of trace 1 such that $p_\lambda c(T_\lambda) p_\lambda \cong L(\F_\infty \ast L^\infty([0,1]))$ and therefore:
    \begin{align*}
        c(M_a) \cong ((L(\F_\infty) \otimes L^\infty([0,1])) \ast L(\F_\infty)) \otimes \cB(\cH) = P 
    \end{align*}
    as required. Thus $c(M_a) \cong P$ for all $a > 0$ and by Lemma \ref{Lemma: non isomorphic factors}, the $M_a$'s are pairwise non-isomorphic. We know already by \cite{Ueda11} that $P$ is full. To show primeness of $P$, it is enough to show primeness of $P_0$ as a II$_1$ factor and this follows from \cite[Corollary 1.3]{ChifanHoudayer}.
\end{proof}

\printbibliography
\end{document}